\documentclass[pdflatex,sn-mathphys-num]{sn-jnl}

\usepackage{amsmath,amssymb,amsfonts}
\usepackage[title]{appendix}

\hypersetup{
  pdftitle={Exact Lagrangian Realization and Robust Strain Sensing in
    Incompressible Flow},
  pdfauthor={Hao Huang},
  pdfsubject={Incompressible Euler and Navier--Stokes equations},
  pdfkeywords={exact Lagrangian realization, deformation gradient,
    Beltrami fields, geometric control, strain sensing}
}

\theoremstyle{thmstyleone}
\newtheorem{theorem}{Theorem}[section]
\newtheorem{lemma}[theorem]{Lemma}
\newtheorem{proposition}[theorem]{Proposition}
\newtheorem{corollary}[theorem]{Corollary}
\theoremstyle{thmstylethree}
\newtheorem{definition}[theorem]{Definition}
\theoremstyle{thmstyletwo}
\newtheorem{remark}[theorem]{Remark}

\newcommand{\R}{\mathbb R}
\newcommand{\T}{\mathbb T}
\newcommand{\Sym}{\operatorname{Sym}}
\newcommand{\Symz}{\operatorname{Sym}_0}
\newcommand{\tr}{\operatorname{tr}}
\newcommand{\Id}{\operatorname{Id}}
\newcommand{\SL}{\operatorname{SL}}
\newcommand{\esssup}{\operatorname*{ess\,sup}}
\newcommand{\ip}[2]{\left\langle #1,#2\right\rangle}
\newcommand{\norm}[1]{\left\lVert #1\right\rVert}
\newcommand{\abs}[1]{\left|#1\right|}
\newcommand{\sphere}{\mathbb S}

\begin{document}

\title[Exact Lagrangian Realization]{Exact Lagrangian Realization and Robust
Strain Sensing in Incompressible Flow}

\author*[1,2]{\fnm{Hao} \sur{Huang}}\email{huanghao92@skku.edu}

\affil*[1]{\orgdiv{Department of Computer Science and Engineering},
  \orgname{Sungkyunkwan University},
  \orgaddress{\city{Suwon-si}, \postcode{16419},
  \state{Gyeonggi-do}, \country{Republic of Korea}}}

\affil[2]{\orgdiv{School of Mathematics and Statistics},
  \orgname{Linyi University},
  \orgaddress{\city{Linyi}, \postcode{276000}, \country{China}}}

\abstract{We prove that the full group \(\SL(3,\R)\) occurs as the set of
one-particle deformation gradients of periodic, unforced, single-shell
solutions of both the three-dimensional Euler and Navier--Stokes equations.
More precisely, given a particle label \(p\in\T^3\), a time \(T>0\), and
\(F_*\in\SL(3,\R)\), every sufficiently large odd integer \(N\) admits a
real-analytic curl eigenfield \(W_N\) with
\(\operatorname{curl}W_N=NW_N\). With an explicit scalar amplitude, \(W_N\)
yields a steady Euler solution; with an explicit exponentially decaying
amplitude, it yields a Navier--Stokes solution for any positive viscosity, and
in both cases \(\nabla_aX(p,T)=F_*\). The lifted particle trajectory is an
embedded analytic arc with nowhere-vanishing velocity. The construction
combines global trace-free symmetric-matrix control on \(\SL(3,\R)\), a Beltrami Cauchy
problem along the controlled arc, Runge approximation, inverse localization on
the torus, and a finite-dimensional endpoint correction.

We also classify finite material-direction systems that determine every
trace-free strain after arbitrary volume-preserving deformation. In dimension
\(n\), this congruence-robust property holds exactly when the associated
rank-one projectors span \(\Sym(n)\); hence the sharp number of scalar channels
is \(n(n+1)/2\). In dimension three, among minimal systems the undeformed
outer-product lower bound is maximized exactly by the six axes of a regular
icosahedron. The realization theorem shows that the full deformation-group
quantifier in this sensing result is dynamically attained within the rigid
class above.}



\maketitle

\section{Introduction}

Let \(u\) be a smooth divergence-free velocity field on the flat torus
\(\T^3_{2\pi}=(\R/2\pi\mathbb Z)^3\), and let
\[
  \partial_tX(a,t)=u(X(a,t),t),
  \qquad
  X(a,0)=a
\]
be its Lagrangian flow.  The deformation gradient
\[
  F(a,t)=\nabla_aX(a,t)
\]
satisfies
\[
  \partial_tF=(\nabla u)(X(a,t),t)F,
  \qquad
  \det F(a,t)=1.
\]
Thus incompressibility confines every one-particle deformation gradient to
\(\SL(3,\R)\).  The principal question of this paper is whether this algebraic
constraint is dynamically sharp under the equations of motion themselves:
given a label, a positive time, and a matrix in \(\SL(3,\R)\), can that matrix be
prescribed as the exact derivative of the material endpoint of an unforced
Euler or Navier--Stokes flow?  We answer this question while keeping the
background geometry and the spectral support of the velocity field fixed in
a particularly rigid form.

We write the Euler and Navier--Stokes equations simultaneously as
\begin{equation}\label{eq:Euler-NS}
  \partial_tu+u\cdot\nabla u+\nabla p=\nu\Delta u,
  \qquad
  \nabla\cdot u=0,
  \qquad
  \nu\ge0,
\end{equation}
with \(\nu=0\) corresponding to Euler.  Our main result is the following
global endpoint theorem.

\begin{theorem}[Global exact Lagrangian realization]
\label{thm:global-SL-realization}
Fix
\[
  p\in\T^3_{2\pi},
  \qquad
  T>0,
  \qquad
  F_*\in\SL(3,\R),
  \qquad
  \nu\ge0.
\]
For every sufficiently large odd integer \(N\), there exists a real-analytic
periodic vector field \(W_N\) satisfying
\[
  \nabla\times W_N=NW_N
\]
such that
\[
  u_N(x,t)=c_Ne^{-\nu N^2t}W_N(x),
  \qquad
  c_N=
  \begin{cases}
    (NT)^{-1},&\nu=0,\\[4pt]
    \displaystyle\frac{\nu N}{1-e^{-\nu N^2T}},&\nu>0,
  \end{cases}
\]
is an exact unforced solution and its Lagrangian flow satisfies
\[
  \nabla_aX_N(p,T)=F_*.
\]
The lifted trajectory issuing from \(p\) is an embedded real-analytic arc
with nowhere-vanishing velocity.  Consequently, for each fixed
\(p\in\T^3_{2\pi}\), \(T>0\), and \(\nu\ge0\), the set of one-particle
deformation gradients at time \(T\) generated by this class of solutions is
exactly \(\SL(3,\R)\).
\end{theorem}

We refer to Theorem~\ref{thm:global-SL-realization} as Theorem A.

The rigidity in Theorem A is simultaneous.  The torus is the fixed flat
torus, the particle label and terminal time are prescribed, no control force
is present, and the velocity remains in one positive curl eigenspace.  The
same construction works for Euler and for every positive constant viscosity;
viscosity changes only the physical clock through the explicit factor
\(c_Ne^{-\nu N^2t}\).  The reverse inclusion in the final equality is forced
by incompressibility, so the theorem identifies the full endpoint range
rather than merely a large subset of it.

The proof is built around an endpoint-regular Beltrami jet carried by a
moving embedded arc.  The finite-dimensional starting point is the
right-invariant control system
\[
  \dot G=SG,
  \qquad
  S\in\Symz(3).
\]
Trace-free symmetric matrices and their commutators generate
\(\mathfrak{sl}(3)\).  We strengthen the resulting accessibility in two
directions needed by the PDE construction: the control path is made regular
for the endpoint map, and it is kept inside a suitable half-space orbit so
that an associated particle arc has a strictly monotone projection.  After a
small compatible perturbation, this gives analytic data
\[
  \dot x=v,\qquad
  \dot F=AF,\qquad
  A=S+\frac12[v]_\times,
\]
with \(F(0)=\Id\), \(F(1)=F_*\), and with an eight-parameter endpoint
derivative that is an isomorphism.  The relation between the skew part of
\(A\) and \(v\) is precisely the first-jet compatibility imposed by
\(\nabla\times V=V\).

An analytic ribbon through the controlled arc then provides Cauchy data for
the Beltrami equation.  The local Cauchy--Kovalevskaya construction and the
Euclidean Runge theorem for Beltrami fields
\cite{EncisoPeraltaSalas2012} produce a global Euclidean field with the
required moving-orbit jet to arbitrary accuracy.  The inverse-localization
theorem on the torus \cite{EncisoPeraltaSalasTorres2017} transfers the
construction to every sufficiently large odd curl frequency.  Approximation
alone would not prescribe the terminal derivative exactly.  Endpoint
regularity is retained throughout the construction, and a
finite-dimensional regular-zero argument is applied after the Runge step and
again after toral localization.  These two corrections recover the exact
matrix \(F_*\) while preserving the embedded trajectory and its nonzero
velocity.

The Lagrangian formulation of incompressible flow goes back to Arnold's
interpretation of the Euler equation as geodesic motion on the group of
volume-preserving diffeomorphisms and to the analytic framework of Ebin and
Marsden \cite{Arnold1966,EbinMarsden1970}.  These works make the flow map and
its derivative intrinsic objects, but do not address their endpoint range.  A
finite-dimensional matrix evolution does arise for affine free-boundary Euler
flows: in the incompressible case the affine deformation follows a geodesic in
\(\SL(3,\R)\) \cite{Sideris2017}.  That reduction concerns an evolving
ellipsoidal domain and a spatially affine ansatz, rather than a prescribed
one-particle endpoint on a fixed periodic background.

There is also an extensive controllability theory for incompressible fluids.
Degenerate Fourier forcing yields Eulerian controllability and approximation
results for Euler and Navier--Stokes systems
\cite{AgrachevSarychev2005}.  For Lagrangian targets, boundary controls have
been used to move curves and particle sets approximately in two- and
three-dimensional Euler flows
\cite{GlassHorsin2010,GlassHorsin2012}, and analogous approximate Lagrangian
controllability results are available for viscous flows with controlled
boundaries \cite{LiaoSueurZhang2022}.  On the three-torus, Nersesyan proved
that a finite-dimensional body force can simultaneously approximately control
the velocity and the full Lagrangian flow, the latter in \(C^1\)
\cite{Nersesyan2015}.  In particular, the distinction from
Theorem~\ref{thm:global-SL-realization} is not simply that the present target
is a first derivative: \(C^1\)-approximation already sees that derivative.
The distinctions are exactness and admissible dynamics.  Theorem A imposes an
exact matrix equality while using no external control and restricting the
solution to one curl eigenspace.

The Beltrami literature provides the local-to-global flexibility on which our
construction builds.  Analytic Cauchy data and Runge approximation have been
used to produce Euclidean Beltrami fields with prescribed, structurally stable
vortex lines and tubes
\cite{EncisoPeraltaSalas2012,EncisoPeraltaSalas2015}; complementary local
normal forms were obtained in \cite{SatoYamada2019}.  For the standard flat
torus, the inverse-localization theorem of
\cite[Theorem~2.1]{EncisoPeraltaSalasTorres2017} is especially relevant: for
each Euclidean field \(v\) with \(\nabla\times v=v\), each finite \(C^m\)
order, and each prescribed accuracy, a rescaled toral curl eigenfield
approximates \(v\) at every sufficiently large odd frequency.  The arithmetic
dependence of inverse localization on a general flat torus has since been
classified, for Laplace eigenfunctions, in
\cite{EncisoGarciaRuizPeraltaSalas2023}.  Inverse localization is nevertheless
an approximation statement.  It does not by itself prescribe an exact
finite-time Lagrangian endpoint.  Here exactness comes from constructing a
compatible Beltrami-jet family that is regular for the endpoint map and then
applying a finite-dimensional regular-zero correction after localization.

Realization in the sense of dynamical universality concerns another target.
Cardona, Miranda, Peralta-Salas and Presas embed non-autonomous dynamics into
steady Euler flows on suitable higher-dimensional Riemannian manifolds; in
particular, orientation-preserving diffeomorphisms occur exactly as
first-return maps on invariant submanifolds
\cite{CardonaMirandaPeraltaSalasPresas2023}.  Berger, Florio and Peralta-Salas
constructed decaying Beltrami fields in \(\R^3\) whose Poincar\'e maps
approximate arbitrary area-preserving disk diffeomorphisms at suitable
sections and scales \cite{BergerFlorioPeraltaSalas2023}.  A return map on an
adapted section or invariant submanifold is different from the derivative of
the material endpoint at a label and physical time fixed in advance.

Against this background, the content of Theorem A is the conjunction of its
constraints.  The ambient space is the standard flat three-torus; the label
and terminal time are prescribed; no body or boundary force is used; the
velocity belongs to one positive curl eigenspace; and the terminal derivative
is exactly an arbitrary specified element of \(\SL(3,\R)\).  Moreover, the
construction works at every sufficiently large odd frequency and, after an
explicit change of clock, for both Euler and every positive constant
viscosity.  To the best of our knowledge, none of the realization or
controllability results above provides this fixed-background, unforced,
single-shell exact endpoint statement.

The second theme of the paper is finite Lagrangian sensing.  Let
\(d\in\sphere^{n-1}\) be a fixed initial material direction and set
\[
  \tau_d(a,t)=\frac{F(a,t)d}{\abs{F(a,t)d}}.
\]
If \(S=(\nabla u+\nabla u^\top)/2\) is the rate-of-strain tensor, then
\begin{equation}\label{eq:stretch}
  \frac{d}{dt}\log\abs{F(a,t)d}
  =
  \tau_d(a,t)^\top S(X(a,t),t)\tau_d(a,t).
\end{equation}
Each initial direction therefore supplies one transported rank-one
observation of the instantaneous strain.  This raises a sharp
finite-dimensional question: which fixed initial direction systems determine
every trace-free symmetric tensor after every volume-preserving deformation?

For
\(\mathcal D=\{d_1,\ldots,d_M\}\subset\sphere^{n-1}\) and
\(F\in\SL(n,\R)\), define
\[
  \tau_j(F)=\frac{Fd_j}{\abs{Fd_j}}.
\]
We call \(\mathcal D\) robustly trace-free sensing if
\[
  B\longmapsto
  \bigl(\tau_j(F)^\top B\tau_j(F)\bigr)_{j=1}^M
\]
is injective on \(\Symz(n)\) for every \(F\in\SL(n,\R)\).

\medskip
\noindent\textbf{Theorem B (sharp finite sensing).}
For \(n\ge2\), a direction system
\(\mathcal D=\{d_1,\ldots,d_M\}\subset\sphere^{n-1}\) is robustly
trace-free sensing if and only if
\[
  \operatorname{span}\{d_j\otimes d_j:1\le j\le M\}=\Sym(n).
\]
Equivalently, the static quadratic measurement map
\[
  B\longmapsto(d_j^\top Bd_j)_{j=1}^M
\]
is injective on the full symmetric space \(\Sym(n)\).  In particular, the
sharp number of fixed material directions is
\[
  N_n=\frac{n(n+1)}2.
\]
The formal statement and proof are given in
Theorem~\ref{thm:classification}.
\medskip

The dimension of \(\Symz(n)\) is \(N_n-1\), but robustness under all
\(F\in\SL(n,\R)\) requires \(N_n\) channels.  The additional channel is not a
dimension-counting artifact: congruence by \(F\) moves the trace-free
hyperplane relative to the fixed rank-one projectors, and uniform
injectivity is equivalent to spanning the whole of \(\Sym(n)\).  Quadratic
measurements and outer-product frames have a well-developed static theory
\cite{BenedettoFickus2003,CasazzaPinkhamTuomanen2016}; the quantifier that
drives Theorem B is robustness after every volume-preserving congruence.
This is also distinct from phase retrieval, where the unknown is a vector
reconstructed from phaseless measurements rather than an arbitrary symmetric
tensor \cite{BalanCasazzaEdidin2006}.

In dimension three, Theorem B gives the minimal number six.  There remains a
quantitative design problem: among all minimal systems, which one maximizes
the worst quadratic sensing strength?  For six unit directions
\(\mathcal D=\{d_1,\ldots,d_6\}\subset\sphere^2\), set
\[
  \alpha_{\mathcal D}
  =
  \inf_{\substack{B\in\Sym(3)\\ \norm{B}_{\mathrm F}=1}}
  \sum_{j=1}^6(d_j^\top Bd_j)^2.
\]

\medskip
\noindent\textbf{Theorem C (the icosahedral optimum).}
Every six-direction system in \(\R^3\) satisfies
\[
  \alpha_{\mathcal D}\le\frac45.
\]
Equality holds if and only if the six unoriented lines are equiangular and
tight:
\[
  \sum_{j=1}^6d_j\otimes d_j=2\Id,
  \qquad
  \abs{d_i\cdot d_j}^2=\frac15
  \quad(i\ne j).
\]
In that case, for every \(B\in\Sym(3)\), writing
\[
  B=B_0+\frac{\tr B}{3}\Id,
  \qquad
  B_0\in\Symz(3),
\]
one has the exact identity
\[
  \sum_{j=1}^6(d_j^\top Bd_j)^2
  =
  \frac45\norm{B_0}_{\mathrm F}^2
  +
  \frac23(\tr B)^2.
\]
The six axes through opposite vertices of a regular icosahedron realize the
equality geometry.  The formal result is Theorem~\ref{thm:optimal-six},
with its icosahedral realization in Corollary~\ref{cor:icosahedron}.
\medskip

The equality case is the six-vector, three-dimensional outer-product frame
optimum \cite{CasazzaPinkhamTuomanen2016}.  Its equiangular geometry is also
the classical optimal configuration of six unoriented lines in \(\R^3\) and
fits the Grassmannian and spherical-design descriptions
\cite{LemmensSeidel1973,ConwayHardinSloane1996,DelsarteGoethalsSeidel1977}.
Here it acquires a Lagrangian interpretation: among minimal robust systems,
the icosahedral axes maximize the undeformed outer-product lower bound.
Lemma~\ref{lem:congruence} transports this bound through an arbitrary
volume-preserving deformation with an explicit condition-number loss.

Theorems A--C close an algebraic--dynamical loop.  Specialized to
\(n=3\), Theorem B requires robustness over the whole group \(\SL(3,\R)\),
while Theorem A proves that this entire test set is dynamically attained by
exact solutions in the rigid periodic unforced single-shell class.
Consequently, in dimension three, the necessity statement and the sharp
six-channel count are dynamically sharp rather than artifacts of an enlarged
kinematic model.  Theorem C then selects the optimal geometry at the minimal
channel count.

The paper develops several further consequences of this loop.  At the
lowest nontrivial curl frequency, an explicit eight-dimensional Beltrami jet
gives a local endpoint chart in \(\SL(3,\R)\), while a closed-form construction
realizes every matrix in
\(\Sym^+(3)\cap\SL(3,\R)\) at a fixed material point.  At sufficiently high odd
frequencies, generic finite point configurations admit simultaneous exact
interpolation of compatible Beltrami first jets, which yields simultaneous
realization of positive-definite determinant-one targets by one exact
unforced solution.  The accompanying dimension, arithmetic-resonance, and
single-shell time obstructions identify the sharp limitations of this
multipoint ansatz.

After completing this main algebraic--dynamical loop, we record several
consequences and sharp limits of the same observation geometry.  They include
an exact six-channel form of the classical Navier--Stokes strain criterion, a
three-direction one-sided fixed-trajectory Euler criterion, and the parallel
stochastic Cauchy consequence for Navier--Stokes
\cite{BusnelloFlandoliRomito2005,ConstantinIyer2008}.  A finite one-sided
obstruction and a moving-pulse kinematic firewall identify where
finite-dimensional reconstruction and volume preservation alone cease to
control fixed-trajectory accumulation.  These later results clarify the PDE
frontier; the principal new mechanism and the main claims of the paper are
Theorems A--C.

The exposition follows the logical construction of these results.  We first
develop the explicit and global Beltrami realization theory, including the
endpoint-regular control lemma, the moving-arc Cauchy construction, and the
multipoint interpolation theorem.  We then prove the congruence
classification and the icosahedral optimum, and derive their dynamical
sharpness.  A consolidated section records continuation consequences, and the
last section isolates the sharp one-sided and fixed-trajectory limitations.

\section{Exact periodic realization of determinant-one deformations}

This section proves Theorem~\ref{thm:global-SL-realization} and develops
complementary low-frequency and multipoint realization results.  We begin
with an explicit first-jet interpolation on the lowest nontrivial Beltrami
shell.  Beltrami fields and their local flexibility have a substantial
literature; see, for example, \cite{SatoYamada2019}.  The point here is to
tie fixed-eigenvalue first-jet interpolation to exact deformation
realization.

\begin{lemma}[Beltrami jet interpolation]
\label{lem:Beltrami}
For every \(H\in\Symz(3)\), there exists a real-analytic vector field \(W_H\)
on \(\T^3_{2\pi}\) such that
\[
  \nabla\cdot W_H=0,\qquad
  \nabla\times W_H=\sqrt2\,W_H,\qquad
  W_H(0)=0,\qquad
  \nabla W_H(0)=H.
\]
\end{lemma}

\begin{proof}
Take
\[
  k_1=(1,1,0),\quad k_2=(1,-1,0),\quad
  k_3=(1,0,1),\quad k_4=(1,0,-1),
\]
and choose
\[
\begin{aligned}
e_1&=(0,0,1),& f_1&=2^{-1/2}(1,-1,0),\\
e_2&=(0,0,1),& f_2&=2^{-1/2}(-1,-1,0),\\
e_3&=(0,1,0),& f_3&=2^{-1/2}(-1,0,1),\\
e_4&=(0,1,0),& f_4&=2^{-1/2}(1,0,1).
\end{aligned}
\]
Then
\[
  k_j\times e_j=\sqrt2\,f_j,\qquad
  k_j\times f_j=-\sqrt2\,e_j.
\]
The real fields
\[
  C_j=e_j\cos(k_j\cdot x)-f_j\sin(k_j\cdot x),
  \qquad
  R_j=e_j\sin(k_j\cdot x)+f_j\cos(k_j\cdot x)
\]
are Beltrami eigenfields with eigenvalue \(\sqrt2\).

The eight-dimensional map sending the coefficients of
\[
  W=\sum_{j=1}^4(\alpha_jC_j+\beta_jR_j)
\]
to
\[
  \bigl(W(0),\operatorname{sym}\nabla W(0)\bigr)
  \in\R^3\oplus\Symz(3)
\]
has determinant \(-\sqrt2\); the matrix is recorded in
Appendix~\ref{app:jet-matrix}.  It is therefore an isomorphism.  Choose the
coefficients so that \(W(0)=0\) and
\(\operatorname{sym}\nabla W(0)=H\).  Since
\(\nabla\times W(0)=\sqrt2W(0)=0\), the antisymmetric part of
\(\nabla W(0)\) vanishes, hence \(\nabla W(0)=H\).
\end{proof}

\begin{theorem}[Local realization of arbitrary volume deformations]
\label{thm:local-SL-realization}
Fix \(a\in\T^3_{2\pi}\), \(T>0\), and \(\nu\ge0\).  There exists an open
neighborhood \(\mathcal U=\mathcal U(a,T,\nu)\) of \(\Id\) in
\(\SL(3,\R)\) such that every \(F_*\in\mathcal U\) is realized as
\[
  \nabla_aX(a,T)=F_*
\]
by an exact real-analytic periodic unforced Euler solution when \(\nu=0\),
and by an exact real-analytic periodic unforced Navier--Stokes solution when
\(\nu>0\).
\end{theorem}

\begin{proof}
Let \(\lambda=\sqrt2\), and let \(\mathcal V_0\) be the
eight-dimensional space spanned by the fields \(C_j,R_j\) in the proof of
Lemma~\ref{lem:Beltrami}.  That proof and
Appendix~\ref{app:jet-matrix} show that
\[
  W\longmapsto
  \bigl(W(0),\operatorname{sym}\nabla W(0)\bigr)
\]
is an isomorphism from \(\mathcal V_0\) onto
\(\R^3\oplus\Symz(3)\).  After translation, the same is true at \(a\) for
\[
  \mathcal V_a
  :=
  \left\{W_0(\,\cdot-a):W_0\in\mathcal V_0\right\}.
\]

For \(W\in\mathcal V_a\), set
\[
  b_\nu(t)=e^{-\nu\lambda^2t},\qquad
  u_W(x,t)=b_\nu(t)W(x),\qquad
  p_W(x,t)=-\frac12b_\nu(t)^2\abs{W(x)}^2.
\]
As in the proof of Theorem~\ref{thm:realization}, the Beltrami and heat
identities make \((u_W,p_W)\) an exact periodic unforced solution of the
required equation.

Let \(X_W\) be its flow and define
\[
  \mathcal E_T:\mathcal V_a\longrightarrow\SL(3,\R),
  \qquad
  \mathcal E_T(W)=\nabla_aX_W(a,T).
\]
Smooth dependence of finite-dimensional ODEs on parameters makes
\(\mathcal E_T\) smooth, and \(\mathcal E_T(0)=\Id\).  If
\(Z\in\mathcal V_a\), linearizing
\[
  \dot X_W=b_\nu(t)W(X_W),\qquad
  \frac{d}{dt}\nabla_aX_W
  =
  b_\nu(t)\nabla W(X_W)\nabla_aX_W
\]
at \(W=0\) gives
\[
  D\mathcal E_T(0)[Z]
  =
  \theta_{\nu,\lambda}(T)\nabla Z(a),
  \qquad
  \theta_{\nu,\lambda}(T)
  :=
  \int_0^T e^{-\nu\lambda^2t}\,dt>0.
\]
The trajectory variation contributes no additional term here: it is
multiplied by the gradient of the zero base field.

With the convention \([v]_\times z=v\times z\), the Beltrami equation
implies
\[
  \nabla Z(a)
  =
  \operatorname{sym}\nabla Z(a)
  +
  \frac{\lambda}{2}[Z(a)]_\times.
\]
The map
\[
  (v,S)\longmapsto S+\frac{\lambda}{2}[v]_\times
\]
is an isomorphism from
\(\R^3\oplus\Symz(3)\) onto \(\mathfrak{sl}(3)\), by the unique
symmetric--skew decomposition.  Hence
\[
  D\mathcal E_T(0):
  \mathcal V_a\longrightarrow T_{\Id}\SL(3,\R)
\]
is an isomorphism.  The inverse function theorem now gives neighborhoods
\(\mathcal O\) of \(0\) and \(\mathcal U\) of \(\Id\) such that
\(\mathcal E_T:\mathcal O\to\mathcal U\) is a diffeomorphism.
\end{proof}

\begin{remark}[Moving-particle version]
\label{rem:local-SL-moving}
The realization in Theorem~\ref{thm:local-SL-realization} can be made at a
particle whose lifted velocity is nowhere zero.  For a realizing field \(W\),
choose a constant \(V\in\R^3\) with
\(\abs V>\norm{W}_{L^\infty}\) and apply the Galilean transformation
\[
  u_{W,V}(x,t)
  =
  V+e^{-\nu\lambda^2t}W(x-Vt).
\]
On the universal cover, its flow satisfies
\[
  X_{W,V}(a,t)=Vt+X_W(a,t),
  \qquad
  \nabla_aX_{W,V}(a,t)=\nabla_aX_W(a,t),
\]
and
\[
  \abs{\dot X_{W,V}(a,t)}
  \ge
  \abs V-\norm W_{L^\infty}>0.
\]
The boosted trajectory is generally not a straight line; it is
\(a+Vt\) only when \(W(a)=0\).
\end{remark}

\begin{theorem}[One-point realization of every SPD volume deformation]
\label{thm:realization}
Let
\[
  F_*\in\Sym^+(3)\cap\SL(3,\R),\qquad T>0,\qquad \nu\ge0.
\]
There exists an exact real-analytic periodic unforced solution of
\eqref{eq:Euler-NS}, interpreted as Navier--Stokes when \(\nu>0\) and as
Euler when \(\nu=0\), such that
\[
  X(0,t)=0\quad(0\le t\le T),
  \qquad
  \nabla_aX(0,T)=F_*.
\]
\end{theorem}

\begin{proof}
Set \(H=\log F_*\in\Symz(3)\), and let \(W_H\) be given by
Lemma~\ref{lem:Beltrami}.  With \(\lambda=\sqrt2\), define
\[
  c_{\nu,T}
  =
  \begin{cases}
    \displaystyle
    \frac{\nu\lambda^2}{1-e^{-\nu\lambda^2T}},
      &\nu>0,\\[6pt]
    T^{-1},&\nu=0,
  \end{cases}
  \qquad
  u(x,t)=c_{\nu,T}e^{-\nu\lambda^2t}W_H(x),
\]
and
\[
  p(x,t)
  =
  -\frac{c_{\nu,T}^2}{2}
  e^{-2\nu\lambda^2t}\abs{W_H(x)}^2.
\]
The heat evolution cancels the viscous term when \(\nu>0\), and is stationary
when \(\nu=0\).  The Beltrami identity gives
\[
  (W_H\cdot\nabla)W_H
  =
  \nabla\frac{\abs{W_H}^2}{2},
\]
so the nonlinearity is cancelled by the pressure.  Thus \((u,p)\) solves
\eqref{eq:Euler-NS}.

Because \(W_H(0)=0\), the origin is a fixed material point.  Along it,
\[
  \nabla u(0,t)=c_{\nu,T}e^{-\nu\lambda^2t}H.
\]
Hence
\[
  F(0,t)
  =
  \begin{cases}
    \displaystyle
    \exp\left(
      \frac{1-e^{-\nu\lambda^2t}}
           {1-e^{-\nu\lambda^2T}}H
    \right),&\nu>0,\\[10pt]
    \exp\left(\frac{t}{T}H\right),&\nu=0,
  \end{cases}
\]
and \(F(0,T)=e^H=F_*\).
\end{proof}

\subsection{Global realization of arbitrary volume-deformation gradients}

The preceding two results use one explicit low-frequency eigenspace.  We now
show that high-frequency Beltrami fields exhaust the whole deformation group
at one moving material particle.  The proof combines geometric control,
the Beltrami Cauchy--Kovalevskaya and Runge theorems, inverse localization on
the torus, and a finite-dimensional endpoint correction.

The relation of this fixed-background endpoint problem to neighboring
realization results was discussed in the Introduction.  We now turn to the
construction.  For an autonomous vector field \(V\), we write
\(\phi_t^V\) for its local flow.

\begin{lemma}[Embedded endpoint-regular control]
\label{lem:embedded-control}
For every \(F_*\in\SL(3,\R)\), there are a jointly real-analytic family
\[
  S_\eta:[0,1]\longrightarrow\Symz(3),
  \qquad
  \eta\in B_\rho(0)\subset\R^8,
\]
a nonzero \(v_0\in\R^3\), and real-analytic families
\[
\begin{aligned}
  \dot v_\eta&=S_\eta v_\eta,
  &v_\eta(0)&=v_0,\\
  A_\eta&=S_\eta+\frac12[v_\eta]_\times,\\
  \dot F_\eta&=A_\eta F_\eta,
  &F_\eta(0)&=\Id,\\
  \dot x_\eta&=v_\eta,
  &x_\eta(0)&=0,
\end{aligned}
\]
such that
\[
  F_0(1)=F_*,
  \qquad
  D_\eta F_\eta(1)\big|_{\eta=0}:
  \R^8\longrightarrow T_{F_*}\SL(3,\R)
\]
is an isomorphism.  Moreover, there are \(e\in\R^3\) and \(m>0\) for which
\[
  e\cdot v_\eta(t)\ge m
  \qquad
  (\eta\in B_\rho(0),\ 0\le t\le1).
\]
The curves \(x_\eta\) are therefore embedded analytic arcs with nowhere
vanishing velocity.  The family can be chosen so that all these arcs lie in
an arbitrarily small ball.
\end{lemma}

\begin{proof}
Write \(\mathfrak s=\Symz(3)\).  We first construct an endpoint-regular
horizontal path for
\[
  \dot G=SG,
  \qquad
  S\in\mathfrak s,
  \qquad
  G(0)=\Id.
\]
Choose \(q\ne0\) such that \(F_*q\) is not a negative multiple of \(q\).
Such a \(q\) exists: otherwise linearity would force \(F_*\) to be a negative
scalar matrix, contradicting \(\det F_*=1\).  With
\[
  e=\frac q{\abs q}+\frac{F_*q}{\abs{F_*q}}
\]
we have \(e\cdot q>0\) and \(e\cdot F_*q>0\).  Thus both endpoints belong to
\[
  \Omega_{e,q}
  =
  \{G\in\SL(3,\R):e\cdot Gq>0\}.
\]
The orbit map
\[
  \pi_q:\SL(3,\R)\longrightarrow\R^3\setminus\{0\},
  \qquad
  \pi_q(G)=Gq,
\]
is a homogeneous fiber bundle.  In coordinates with \(q=e_1\), its fiber is
\(\R^2\rtimes\SL(2,\R)\), hence is connected.  Its restriction over the
contractible half-space \(\{y:e\cdot y>0\}\) is therefore trivial, so
\(\Omega_{e,q}\) is connected.

The right-invariant distribution
\[
  \mathcal H_G=\{SG:S\in\mathfrak s\}
\]
is bracket generating.  Indeed, if
\(D=\operatorname{diag}(d_1,d_2,d_3)\in\mathfrak s\) has distinct diagonal
entries, then
\[
  [D,E_{ij}+E_{ji}]
  =
  (d_i-d_j)(E_{ij}-E_{ji}).
\]
Thus \([\mathfrak s,\mathfrak s]\) contains
\(\mathfrak{so}(3)\), while
\(\mathfrak s\oplus\mathfrak{so}(3)=\mathfrak{sl}(3)\).  The
Chow--Rashevskii theorem \cite{AgrachevSachkov2004}, applied on the connected
open manifold \(\Omega_{e,q}\), gives a piecewise smooth horizontal path from
\(\Id\) to \(F_*\) lying entirely in \(\Omega_{e,q}\).

We next make the endpoint regular without losing the half-space constraint.
Take
\[
  D=\delta\operatorname{diag}(d_1,d_2,d_3)\in\mathfrak s,
  \qquad
  d_i\ne d_j\quad(i\ne j),
\]
where \(\delta>0\) is small.  Four constant pulses, in chronological order
\(L,D,K,-D\), have endpoint
\[
  \Phi(K,L)=e^{-D}e^Ke^De^L,
  \qquad
  K,L\in\mathfrak s.
\]
At \((K,L)=(0,0)\) this is an identity loop, and
\[
  D\Phi(0,0)[K,L]
  =
  \operatorname{Ad}_{e^{-D}}K+L.
\]
For \(i\ne j\),
\[
  \operatorname{Ad}_{e^{-D}}(E_{ij}+E_{ji})
  =
  e^{-\delta(d_i-d_j)}E_{ij}
  +
  e^{\delta(d_i-d_j)}E_{ji}.
\]
Together with \(E_{ij}+E_{ji}\), these matrices span both \(E_{ij}\) and
\(E_{ji}\).  Including the two diagonal trace-free directions gives
\[
  \operatorname{Ad}_{e^{-D}}\mathfrak s+\mathfrak s
  =
  \mathfrak{sl}(3).
\]
Choose eight pairs
\[
  (K_j,L_j)\in\mathfrak s\times\mathfrak s,
  \qquad 1\le j\le8,
\]
whose images under \(D\Phi(0,0)\) form a basis of
\(\mathfrak{sl}(3)\).  For \(\delta\) and the parameters sufficiently small,
the pulse family remains in \(\Omega_{e,q}\).  Insert the base identity loop
before the horizontal path to \(F_*\).  Propagation through the remaining
path maps endpoint variations by an invertible linear map.  Compactness and
the strict half-space inequality along the following path also keep the
whole concatenated small-parameter family inside \(\Omega_{e,q}\).
Consequently the resulting eight-parameter endpoint derivative is an
isomorphism.

Let \(\bar S\) denote the resulting piecewise smooth control, let
\(H_1,\ldots,H_8\) denote its regular variations, and let \(\bar G\) be the
base path.  There is a strict margin
\begin{equation}\label{eq:piecewise-halfspace-margin}
  \min_{0\le t\le1}e\cdot\bar G(t)q
  =:3m_0>0.
\end{equation}
First mollify \(\bar S,H_1,\ldots,H_8\) in \(L^1(0,1)\), preserving their
values in \(\mathfrak s\), and then approximate the resulting smooth controls
uniformly by \(\mathfrak s\)-valued polynomials
\[
  P_0,P_1,\ldots,P_8.
\]
Equivalently, the polynomials may be chosen directly \(L^1\)-close to the
piecewise controls.  Let \(G_\theta\) be driven by
\[
  P_0+\sum_{j=1}^8\theta_jP_j.
\]
The integral equations for \(G_\theta\) and its parameter derivatives,
together with Gronwall's inequality, show that, as all nine approximation
errors tend to zero,
\[
  G_0(1)\longrightarrow F_*,
  \qquad
  D_\theta G_\theta(1)\big|_{\theta=0}
  \longrightarrow
  D\mathcal P_{\bar S}(H_1,\ldots,H_8),
\]
where \(\mathcal P\) is the control endpoint map.  The limit on the right is
an isomorphism.  In a chart of \(\SL(3,\R)\) centered at \(F_*\), the
quantitative inverse function theorem therefore supplies
\(\theta_*\to0\) such that \(G_{\theta_*}(1)=F_*\).  Recenter at
\(\theta_*\), and denote the resulting polynomial family by \(S_\theta\).
Then
\begin{equation}\label{eq:analytic-affine-endpoint}
  G_0(1)=F_*,
  \qquad
  D_\theta G_\theta(1)\big|_{\theta=0}:
  \R^8\longrightarrow T_{F_*}\SL(3,\R)
  \quad\hbox{is an isomorphism}.
\end{equation}
The same \(L^1\) stability, \eqref{eq:piecewise-halfspace-margin}, and a
shrinking of the parameter ball give
\begin{equation}\label{eq:control-halfspace-margin}
  e\cdot G_\theta(t)q\ge2m_0
  \qquad
  (|\theta|<\rho_0,\ 0\le t\le1).
\end{equation}
The family \(S_\theta\) is jointly analytic because all the \(P_j\) are
polynomials.

We finally impose the skew part forced by a moving unit-eigenvalue Beltrami
trajectory.  For \((\theta,\varepsilon)\) near \((0,0)\), define
\[
  v_{\theta,\varepsilon}(t)=\varepsilon G_\theta(t)q,
  \qquad
  A_{\theta,\varepsilon}
  =
  S_\theta+\frac12[v_{\theta,\varepsilon}]_\times,
\]
and let
\[
  \dot F_{\theta,\varepsilon}
  =
  A_{\theta,\varepsilon}F_{\theta,\varepsilon},
  \qquad
  F_{\theta,\varepsilon}(0)=\Id.
\]
At \(\varepsilon=0\), \(F_{\theta,0}=G_\theta\).  Hence, in a chart
\(\chi\) with \(\chi(F_*)=0\), the map
\[
  \mathcal F(\theta,\varepsilon)
  =
  \chi(F_{\theta,\varepsilon}(1))
\]
satisfies
\[
  \mathcal F(0,0)=0,
  \qquad
  D_\theta\mathcal F(0,0)\in\operatorname{GL}(8,\R).
\]
The analytic implicit function theorem yields an analytic
\(\theta=\beta(\varepsilon)\), with \(\beta(0)=0\), such that
\[
  F_{\beta(\varepsilon),\varepsilon}(1)=F_*.
\]
Fix a sufficiently small \(\varepsilon>0\) and recenter the eight parameters
at \(\beta(\varepsilon)\).  Invertibility of the endpoint derivative
persists.  After decreasing \(\rho\),
\eqref{eq:control-halfspace-margin} gives, for every \(|\eta|<\rho\),
\[
  e\cdot v_\eta(t)
  =
  \varepsilon
  e\cdot G_{\beta(\varepsilon)+\eta}(t)q
  \ge
  \varepsilon m_0
  =:m>0.
\]
Thus, if \(0\le t_1<t_2\le1\),
\[
  e\cdot\bigl(x_\eta(t_2)-x_\eta(t_1)\bigr)
  =
  \int_{t_1}^{t_2}e\cdot v_\eta(s)\,ds
  \ge
  m(t_2-t_1)>0.
\]
Every \(x_\eta\) is consequently injective and regular, hence an embedded
analytic arc, and its velocity never vanishes.  Finally,
\[
  \sup_{|\eta|<\rho}\sup_{0\le t\le1}\abs{x_\eta(t)}
  \le
  \varepsilon
  \sup_{\substack{|\theta|<\rho_0\\0\le t\le1}}
  \norm{G_\theta(t)}\abs q.
\]
Choosing \(\varepsilon\) smaller places the whole family in any prescribed
ball.
\end{proof}

\begin{proposition}[Parametric Beltrami extension of the control arc]
\label{lem:curve-Beltrami}
For the family furnished by Lemma~\ref{lem:embedded-control}, after decreasing
\(\rho\) there exist a common neighborhood \(U\) of the arcs
\(x_\eta([0,1])\) and a vector field
\[
  (\eta,y)\longmapsto V_\eta(y)
\]
that is jointly real-analytic on \(B_\rho(0)\times U\) and satisfies
\[
  \nabla\times V_\eta=V_\eta,
  V_\eta(x_\eta(t))=v_\eta(t),
  \qquad
  \nabla V_\eta(x_\eta(t))=A_\eta(t)
  \quad(0\le t\le1).
\]
Consequently,
\[
  x_\eta(t)=\phi^{V_\eta}_t(0)
\]
and
\[
  D\phi^{V_\eta}_1(0)=F_\eta(1).
\]
\end{proposition}

\begin{proof}
All the data extend analytically to a common open interval containing
\([0,1]\).  The half-space estimate in
Lemma~\ref{lem:embedded-control} gives a uniform lower bound for
\(\abs{v_\eta}\).  Put
\[
  T_\eta=\frac{v_\eta}{\abs{v_\eta}}.
\]
Because \(v_\eta(0)=v_0\) is independent of \(\eta\), choose one unit vector
\(E_*\perp T_\eta(0)\) and solve, jointly in \((\eta,t)\),
\begin{equation}\label{eq:adapted-ribbon-frame}
  E_\eta'
  =
  -\ip{E_\eta}{T_\eta'}T_\eta
  +
  \ip{N_\eta}{A_\eta E_\eta}N_\eta,
  \qquad
  E_\eta(0)=E_*,
  \qquad
  N_\eta=T_\eta\times E_\eta.
\end{equation}
The equation preserves \(\abs{E_\eta}=1\) and
\(E_\eta\perp T_\eta\).  Extend \(t\) slightly past the endpoints and then
use strict monotonicity of \(e\cdot x_\eta(t)\), compactness, and the
tubular-neighborhood theorem.  After shrinking \(\rho\), there is one
\(s_0>0\) such that
\[
  \sigma_\eta(t,s)=x_\eta(t)+sE_\eta(t),
  \qquad
  \abs s<s_0,
\]
is a jointly analytic family of embedded oriented ribbons.  Along the central
curve its unit normal is \(N_\eta\), and the mixed second fundamental form
satisfies
\begin{equation}\label{eq:ribbon-mixed-II}
\begin{aligned}
  \operatorname{II}_\eta(E_\eta,v_\eta)
  &=
  \ip{\partial_s\partial_t\sigma_\eta}{N_\eta}\\
  &=
  \ip{E_\eta'}{N_\eta}
  =
  \ip{A_\eta E_\eta}{N_\eta}.
\end{aligned}
\end{equation}

On the ribbon set
\[
  f_\eta(t,s)
  =
  \int_0^t\abs{v_\eta(r)}^2\,dr
  +
  \frac12
  \ip{A_\eta(t)E_\eta(t)}{E_\eta(t)}s^2,
  \qquad
  w_\eta=\nabla_{\Sigma_\eta}f_\eta.
\]
Then
\[
  w_\eta(x_\eta(t))=v_\eta(t),
  \qquad
  j_{\Sigma_\eta}^*(w_\eta^\flat)=df_\eta,
\]
and hence
\[
  d\bigl(j_{\Sigma_\eta}^*w_\eta^\flat\bigr)=0.
\]

We next verify the prescribed first jet.  Differentiation along the central
curve gives
\[
  D_{v_\eta}w_\eta=v_\eta'=A_\eta v_\eta.
\]
For the other tangential column, the surface Hessian and
\eqref{eq:ribbon-mixed-II} yield
\begin{align*}
  \ip{D_{E_\eta}w_\eta}{E_\eta}
  &=
  \ip{A_\eta E_\eta}{E_\eta},\\
  \ip{D_{E_\eta}w_\eta}{v_\eta}
  &=
  \operatorname{Hess}_{\Sigma_\eta}
  f_\eta(E_\eta,v_\eta)\\
  &=
  \ip{A_\eta v_\eta}{E_\eta}
  =
  \ip{A_\eta E_\eta}{v_\eta},\\
  \ip{D_{E_\eta}w_\eta}{N_\eta}
  &=
  \operatorname{II}_\eta(E_\eta,v_\eta)
  =
  \ip{A_\eta E_\eta}{N_\eta}.
\end{align*}
The penultimate equality in the second line uses
\[
  \operatorname{skew}A_\eta
  =
  \frac12[v_\eta]_\times,
  \qquad
  [v_\eta]_\times v_\eta=0.
\]
Thus the two tangential columns of the desired ambient derivative are exactly
those of \(A_\eta\).

The Cauchy theorem for Beltrami fields
\cite[Theorem~3.1]{EncisoPeraltaSalas2012} states that an analytic tangent
field \(w\) on an embedded oriented analytic surface \(\Sigma\), with
\(d(j_\Sigma^*w^\flat)=0\), extends uniquely to an analytic field \(V\) near
\(\Sigma\) satisfying
\[
  \nabla\times V=V,
  \qquad
  V|_\Sigma=w.
\]
Apply it to \((\Sigma_\eta,w_\eta)\).  If
\(B_\eta=\nabla V_\eta\) on the central curve, then
\[
  (B_\eta-A_\eta)v_\eta=0,
  \qquad
  (B_\eta-A_\eta)E_\eta=0.
\]
Moreover, the Beltrami equation and its divergence imply
\[
  \operatorname{skew}B_\eta
  =
  \frac12[v_\eta]_\times
  =
  \operatorname{skew}A_\eta,
  \qquad
  \tr B_\eta=0=\tr A_\eta.
\]
Hence \(B_\eta-A_\eta\) is symmetric, annihilates the tangent plane, and has
zero trace; therefore it vanishes.

It remains to justify the common domain and parameter dependence, which are
not part of the pointwise statement just cited.  Extend \(N_\eta\)
analytically over the ribbon and use the tubular maps
\[
  \Theta_\eta(t,s,n)
  =
  \sigma_\eta(t,s)+nN_\eta(t,s).
\]
The shrinking is performed in the following order.  First choose a
compactly contained parameter ball and a common complex neighborhood of the
slightly extended \(t\)-interval on which the control data are bounded.
Next choose \(s_0\) so that all complexified ribbon maps are nonsingular on
one fixed strip.  On that strip use the analytic sup norm
\[
  \norm{h}_{\varrho}
  =
  \sup_{\substack{|\operatorname{Im}t|<\varrho\\
                   |\operatorname{Im}s|<\varrho}}
  \abs{h(t,s)}
\]
for the pulled-back coefficients and Cauchy data.  In signed-normal
coordinates, the curl equation together with its divergence consequence has
the Cauchy--Kovalevskaya form
\begin{equation}\label{eq:parametric-CK-normal-form}
  M_\eta(t,s,n)\partial_nU_\eta
  =
  \mathcal G_\eta
  \bigl(t,s,n,U_\eta,\partial_tU_\eta,\partial_sU_\eta\bigr).
\end{equation}
The matrices and the right-hand side are jointly analytic.  At \(n=0\) the
surface is noncharacteristic; compactness and the common ribbon
parametrization give the uniform bound
\[
  \inf_{\eta,t,s}
  \abs{\det M_\eta(t,s,0)}
  \ge c_0>0.
\]
After decreasing the complex strip once, the same bound holds throughout a
fixed normal slab.

The standard majorant-series proof applied with the preceding analytic norms
now gives one normal radius \(n_0>0\), independent of \(\eta\), and a solution
\(U_\eta\) jointly analytic in \((\eta,t,s,n)\).  In particular, the first
parameter derivative \(Z_j=\partial_{\eta_j}U_\eta\) satisfies the linearized
Cauchy system obtained from \eqref{eq:parametric-CK-normal-form}:
\begin{align*}
  M_\eta\partial_nZ_j
  &=
  D_U\mathcal G_\eta\,Z_j
  +
  D_{\partial_tU}\mathcal G_\eta\,\partial_tZ_j
  +
  D_{\partial_sU}\mathcal G_\eta\,\partial_sZ_j\\
  &\quad
  +
  \partial_{\eta_j}\mathcal G_\eta
  -
  (\partial_{\eta_j}M_\eta)\partial_nU_\eta,
\end{align*}
with the differentiated analytic Cauchy datum at \(n=0\).  The same
majorant bounds control this system uniformly.

Finally shrink \(\rho\) so that one fixed physical tube \(U\) around the base
arc is contained in every image
\(\Theta_\eta(\{|s|<s_0,\ |n|<n_0\})\) and contains every central arc.
Uniqueness patches the local solutions along the strip, and composition with
\(\Theta_\eta^{-1}\) proves joint analyticity on
\(B_\rho(0)\times U\).

Finally, \(x_\eta'=V_\eta(x_\eta)\), and both
\(D\phi^{V_\eta}_t(0)\) and \(F_\eta(t)\) solve
\[
  \dot G=A_\eta(t)G,
  \qquad
  G(0)=\Id.
\]
Uniqueness gives the last assertion.
\end{proof}

\begin{lemma}[Regular zeros and deformation endpoints]
\label{lem:regular-endpoint-stability}
Let \(B_r\subset\R^8\), let
\(K_0\Subset\operatorname{int}K\subset\R^3\), and let
\[
  Y\in C^1\bigl(B_r;C^2(K;\R^3)\bigr)
\]
be a family of divergence-free vector fields.  Assume that, for every
\(\eta\in B_r\), the \(Y_\eta\)-trajectory from the origin is defined on
\([0,1]\) and remains in \(K_0\).  Fix an analytic chart
\[
  \chi:\mathcal V\subset\SL(3,\R)\longrightarrow\R^8
\]
containing all the deformation endpoints under consideration, and set
\[
  \mathcal E_Y(\eta)
  =
  \chi\bigl(D\phi^{Y_\eta}_1(0)\bigr).
\]
After decreasing \(r\), the map
\[
  Y\longmapsto\mathcal E_Y
\]
is continuous from \(C^1(B_r;C^2(K))\) to
\(C^1(B_r;\R^8)\).  In particular, if
\[
  \mathcal E_Y(0)=0,
  \qquad
  D\mathcal E_Y(0)\in\operatorname{GL}(8,\R),
\]
then every sufficiently small divergence-free perturbation
\(\widetilde Y\) in \(C^1(B_r;C^2(K))\) has a unique parameter
\(\widetilde\eta\) in a fixed smaller ball such that
\[
  \mathcal E_{\widetilde Y}(\widetilde\eta)=0.
\]
Moreover, \(\widetilde\eta\to0\) as \(\widetilde Y\to Y\), and
\(D\mathcal E_{\widetilde Y}(\widetilde\eta)\) remains invertible.
\end{lemma}

\begin{proof}
Write
\[
  z_\eta(t)=\phi^{Y_\eta}_t(0),
  \qquad
  G_\eta(t)=D\phi^{Y_\eta}_t(0).
\]
They solve
\[
  \dot z_\eta=Y_\eta(z_\eta),
  \qquad
  \dot G_\eta=DY_\eta(z_\eta)G_\eta,
  \qquad
  z_\eta(0)=0,
  \quad
  G_\eta(0)=\Id.
\]
For
\[
  Q_i=\partial_{\eta_i}z_\eta,
  \qquad
  H_i=\partial_{\eta_i}G_\eta,
\]
differentiation gives
\begin{align*}
  \dot Q_i
  &=
  DY_\eta(z_\eta)Q_i
  +
  (\partial_{\eta_i}Y_\eta)(z_\eta),\\
  \dot H_i
  &=
  DY_\eta(z_\eta)H_i
  +
  D^2Y_\eta(z_\eta)[Q_i]G_\eta\\
  &\quad
  +
  D(\partial_{\eta_i}Y_\eta)(z_\eta)G_\eta,
\end{align*}
with zero initial data for \(Q_i\) and \(H_i\).  The buffer
\(K_0\Subset\operatorname{int}K\), continuous dependence for ODEs, and
Gronwall's inequality therefore imply, uniformly on a smaller parameter
ball,
\[
  \norm{\mathcal E_{\widetilde Y}-\mathcal E_Y}_{C^1}
  \longrightarrow0
  \quad\hbox{whenever}\quad
  \norm{\widetilde Y-Y}_{C^1(B_r;C^2(K))}
  \longrightarrow0.
\]
The same estimates keep all perturbed trajectories in \(K\) and all
deformation endpoints inside the chart.

For completeness, put \(A=D\mathcal E_Y(0)\).  Choose \(0<r_0<r\) so that
\[
  \sup_{|\eta|\le r_0}
  \norm{
    A^{-1}\bigl(D\mathcal E_Y(\eta)-A\bigr)
  }
  \le\frac14.
\]
For \(\widetilde Y\) sufficiently close to \(Y\), the map
\[
  \mathcal T(\eta)
  =
  \eta-A^{-1}\mathcal E_{\widetilde Y}(\eta)
\]
is a contraction of \(\overline B_{r_0}\) into itself.  Its unique fixed point
\(\widetilde\eta\) is the required zero.  The same estimates give
\(\widetilde\eta\to0\) and preserve invertibility of the derivative.
\end{proof}

\begin{proof}[Proof of Theorem~\ref{thm:global-SL-realization}]
Apply Lemma~\ref{lem:embedded-control} and
Proposition~\ref{lem:curve-Beltrami}, and choose a compact tubular neighborhood
\[
  K\Subset U\cap B_1(0)
\]
of the base arc, with a uniform buffer for the nearby family; here we used the
small-ball clause in Lemma~\ref{lem:embedded-control}.  The set \(K\) can be
chosen diffeomorphic to a closed three-ball, so \(\R^3\setminus K\) is
connected.  Put
\[
  Z_0=V_0,\qquad
  Z_j=\partial_{\eta_j}V_\eta\big|_{\eta=0}
  \quad(1\le j\le8).
\]
All nine fields solve \(\nabla\times Z_j=Z_j\).  The Euclidean Beltrami
Runge theorem \cite[Theorem~3.6]{EncisoPeraltaSalas2012} approximates them,
to arbitrary accuracy in \(C^2(K)\), by global Beltrami fields
\[
  \widetilde Z_0,\ldots,\widetilde Z_8
  \quad\hbox{on }\R^3.
\]

The local affine family
\[
  Z_\eta=Z_0+\sum_{j=1}^8\eta_jZ_j
\]
has at \(\eta=0\) the same deformation endpoint and the same endpoint
derivative as \(V_\eta\).  Its endpoint is therefore \(F_*\), and its endpoint
derivative is an isomorphism.  For the global affine family
\[
  \widetilde Z_\eta
  =
  \widetilde Z_0+\sum_{j=1}^8\eta_j\widetilde Z_j,
\]
choose an analytic chart
\[
  \chi:\mathcal V\subset\SL(3,\R)\longrightarrow\R^8,
  \qquad
  \chi(F_*)=0,
\]
and write
\[
\begin{aligned}
  \mathcal E(\eta)
  &=
  \chi\bigl(D\phi^{Z_\eta}_1(0)\bigr),\\
  \widetilde{\mathcal E}(\eta)
  &=
  \chi\bigl(D\phi^{\widetilde Z_\eta}_1(0)\bigr).
\end{aligned}
\]
Choose the Runge approximants so that
\[
  \max_{0\le j\le8}
  \norm{\widetilde Z_j-Z_j}_{C^2(K)}
  <\delta.
\]
For affine families this implies, on every fixed small parameter ball,
\[
  \norm{\widetilde Z-Z}_{C^1(B_r;C^2(K))}
  \le C_r\delta.
\]
Lemma~\ref{lem:regular-endpoint-stability} therefore gives, for \(\delta\)
sufficiently small, a parameter \(\eta_*\) such that
\[
  \widetilde{\mathcal E}(\eta_*)=0.
\]
After recentering at \(\eta_*\),
\[
  \widetilde{\mathcal E}(0)=0,
  \qquad
  D\widetilde{\mathcal E}(0)
  \in\operatorname{GL}(8,\R).
\]
Notice that the \(D^2Z_0\) term in the parameter variation of the deformation
equation is precisely why the Runge approximation is taken in \(C^2\).
Moreover, compactness and the strict inequality in
Lemma~\ref{lem:embedded-control} give a spacetime tube of the base orbit on
which \(e\cdot Z_0\) is uniformly positive.  After shrinking the parameter
ball and the Runge error, all nearby trajectories remain in this tube and
\(e\cdot\widetilde Z_\eta\) remains uniformly positive there.

The toral part of the inverse-localization theorem
\cite[Theorem~2.1]{EncisoPeraltaSalasTorres2017} says the following.  If
\(Z\) is an entire real field satisfying \(\nabla\times Z=Z\), then for every
\(\epsilon>0\) and integer \(m\ge0\), each sufficiently large odd \(N\)
admits a real toral field \(W_N\), with
\(\nabla\times W_N=NW_N\), such that
\[
  \norm{
    W_N\left(p+\frac{\cdot}{N}\right)-Z
  }_{C^m(B_1)}
  <\epsilon.
\]
Since \(K\Subset B_1\), apply this statement separately to the nine fixed
fields \(\widetilde Z_0,\ldots,\widetilde Z_8\), take \(m=2\), and take the
maximum of the nine thresholds.  Thus, for every sufficiently large odd
\(N\), there are fields at the same prescribed frequency,
\[
  \nabla\times W_{j,N}=N W_{j,N}
  \qquad(0\le j\le8)
\]
such that, simultaneously,
\[
  W_{j,N}\left(p+\frac{\cdot}{N}\right)
  \longrightarrow
  \widetilde Z_j
  \quad\hbox{in }C^2(K).
\]
To make the convergence hold along all sufficiently large odd integers,
rather than only along a selected sequence, take errors
\(\epsilon_k\downarrow0\), take the increasing common thresholds \(N_k\), and
use accuracy \(\epsilon_k\) whenever \(N_k\le N<N_{k+1}\).

Define
\[
\begin{aligned}
  W_{\eta,N}
  &=
  W_{0,N}+\sum_{j=1}^8\eta_jW_{j,N},\\
  \widehat Z_{\eta,N}(y)
  &=
  W_{\eta,N}\left(p+\frac yN\right),\\
  \mathcal E_N(\eta)
  &=
  \chi\bigl(D\phi^{\widehat Z_{\eta,N}}_1(0)\bigr).
\end{aligned}
\]
The common \(C^2(K)\) error tending to zero gives
\[
  \norm{
    \widehat Z_{\cdot,N}-\widetilde Z_{\cdot}
  }_{C^1(B_r;C^2(K))}
  \longrightarrow0
\]
on a fixed small parameter ball.  A second application of
Lemma~\ref{lem:regular-endpoint-stability} gives
\(\eta_N\to0\) with
\[
  \mathcal E_N(\eta_N)=0.
\]
Set
\[
  W_N=W_{\eta_N,N}.
\]
The same spacetime-tube argument, now using
\(\widehat Z_{\eta,N}\to\widetilde Z_\eta\) in \(C^1\), keeps the rescaled
trajectory in \(K\) and preserves the strict half-space velocity margin.  For
all sufficiently large \(N\), its physical image lies in the torus coordinate
ball \(p+K/N\).  Hence the physical trajectory cannot wrap around the torus;
it is embedded and has nonzero velocity.

It remains only to choose the physical clock.  Let
\[
  c_N
  =
  \begin{cases}
    \displaystyle\frac1{NT},&\nu=0,\\[8pt]
    \displaystyle
    \frac{\nu N}{1-e^{-\nu N^2T}},&\nu>0.
  \end{cases}
\]
Since \(\Delta W_N=-N^2W_N\) and the Beltrami nonlinearity is a gradient,
\[
  u_N(x,t)=c_N e^{-\nu N^2t}W_N(x),
  \qquad
  p_N(x,t)
  =
  -\frac12c_N^2e^{-2\nu N^2t}\abs{W_N(x)}^2
\]
is the asserted exact unforced solution.  Put
\[
  a_N(t)=c_Ne^{-\nu N^2t},
  \qquad
  \tau_N(t)=N\int_0^ta_N(s)\,ds.
\]
If \(\widehat\phi_N^\tau\) denotes the flow of
\(\widehat Z_{\eta_N,N}\), then, as long as the rescaled trajectory remains in
the coordinate tube,
\[
  X_N\left(p+\frac{y_0}{N},t\right)
  =
  p+\frac1N\widehat\phi_N^{\tau_N(t)}(y_0).
\]
The displayed choice of \(c_N\) gives
\[
  \tau_N(T)=1,
  \qquad
  \tau_N'(t)>0.
\]
Differentiating the conjugacy at \(y_0=0\) yields
\[
  \nabla_aX_N(p,T)
  =
  D\widehat\phi_N^1(0)
  =
  F_*.
\]
Because \(\tau_N\) is strictly increasing, the strict half-space monotonicity
of the rescaled orbit is preserved by physical time.  The physical orbit lies
in the injective coordinate ball \(p+K/N\), and is therefore embedded, with
\[
  u_N(X_N(p,t),t)\ne0
  \qquad(0\le t\le T).
\]

Conversely, Liouville's formula and \(\nabla\cdot u=0\) give
\[
  \det\nabla_aX(a,T)=1
\]
for every solution in this class.  Hence the endpoint set is exactly
\(\SL(3,\R)\).
\end{proof}

\subsection{Generic multipoint interpolation and its sharp obstructions}

The preceding construction uses only a small part of the lowest nontrivial
Fourier shell.  High-energy Beltrami eigenspaces have enough multiplicity to
interpolate finitely many first jets simultaneously.  The proof combines the
Euclidean Beltrami Runge theorem of
\cite{EncisoPeraltaSalas2012} with the toral inverse-localization theorem of
\cite{EncisoPeraltaSalasTorres2017}.

For a positive curl eigenvalue \(\lambda\), set
\[
  E_\lambda^+
  =
  \left\{
    W\in C^\omega(\T^3;\R^3):
    \nabla\times W=\lambda W
  \right\}.
\]
Write
\[
  \operatorname{Conf}_k(\T^3)
  =
  \left\{
    (x_1,\ldots,x_k)\in(\T^3)^k:
    x_i\ne x_j\ \text{for }i\ne j
  \right\}
\]
for the ordered configuration space, a connected real-analytic manifold.
For \(\boldsymbol x\in\operatorname{Conf}_k(\T^3)\), define
\[
  \mathcal J_{\lambda,\boldsymbol x}W
  =
  \left(
    W(x_\ell),
    \operatorname{sym}\nabla W(x_\ell)
  \right)_{\ell=1}^k
  \in
  \left(\R^3\oplus\Symz(3)\right)^k.
\]

\begin{theorem}[Generic multipoint Beltrami jet interpolation]
\label{thm:multipoint-jet}
Fix \(k\ge1\).  For every sufficiently large odd integer \(\lambda\), there
is a canonical rank-deficient set
\[
  \Sigma_{\lambda,k}
  :=
  \left\{
    \boldsymbol x\in\operatorname{Conf}_k(\T^3):
    \operatorname{rank}\mathcal J_{\lambda,\boldsymbol x}<8k
  \right\}.
\]
This is a proper real-analytic subset with
\[
  \dim_{\mathrm H}\Sigma_{\lambda,k}\le 3k-1.
\]
In particular, it has measure zero and empty interior, and
\(\mathcal J_{\lambda,\boldsymbol x}\) is surjective whenever
\(\boldsymbol x\notin\Sigma_{\lambda,k}\).  For every such
\(\boldsymbol x\) and every
\(H_1,\ldots,H_k\in\Symz(3)\), there is a
\(W\in E_\lambda^+\) satisfying
\[
  W(x_\ell)=0,
  \qquad
  \nabla W(x_\ell)=H_\ell
  \quad(1\le\ell\le k).
\]
\end{theorem}

\begin{proof}
We first prove that global solutions of
\(\nabla\times U=U\) on \(\R^3\) realize arbitrary compatible first jets at
finitely many points.  If \(W\) is one of the entire periodic fields used in
Lemma~\ref{lem:Beltrami}, then \(U(y)=W(y/\sqrt2)\) satisfies
\(\nabla_y\times U=U\); the derivative coordinates of its jet differ only by
the nonzero factor \(1/\sqrt2\).  Hence the eight-dimensional isomorphism in
that lemma, followed by translation, realizes an arbitrary pair
\[
  (U(y),\operatorname{sym}\nabla U(y))
  \in\R^3\oplus\Symz(3)
\]
locally at one point.  Choose mutually disjoint small closed balls around
distinct \(y_1,\ldots,y_k\) in the Euclidean unit ball, and prescribe the
desired local field on each component.  This finite union is a locally finite
union of pairwise disjoint compact sets, each with connected complement, so
the Beltrami Runge theorem
\cite[Theorem~3.6]{EncisoPeraltaSalas2012} approximates this disconnected
local field arbitrarily well in \(C^1\) by a global Euclidean Beltrami field.
The image of the global \(k\)-point jet map is a linear subspace of the
finite-dimensional target
\((\R^3\oplus\Symz(3))^k\).  Since the image is dense, it is the whole target.
We may therefore choose global Euclidean fields
\(U_1,\ldots,U_{8k}\) whose \(k\)-point jets form a basis.

Fix a point \(p\in\T^3\).  The toral inverse-localization theorem
\cite[Theorem~2.1]{EncisoPeraltaSalasTorres2017} implies that, for every
sufficiently large odd integer \(\lambda\), each \(U_s\) has a counterpart
\(W_s\in E_\lambda^+\) for which
\[
  W_s\left(p+\frac{y}{\lambda}\right)
\]
is arbitrarily close to \(U_s(y)\) in \(C^1\) on the Euclidean unit ball.
Take the finitely many approximations with a common sufficiently small error.
At
\[
  x_\ell^{(0)}
  =
  p+\frac{y_\ell}{\lambda},
\]
form the physical jet matrix of \(W_1,\ldots,W_{8k}\), and multiply all of
its derivative rows by \(1/\lambda\).  The resulting normalized matrix is
arbitrarily close to the Euclidean basis matrix because
\[
  \nabla_y
  \left[
    W_s\left(p+\frac{y}{\lambda}\right)
  \right]
  =
  \frac1\lambda
  \nabla_xW_s\left(p+\frac{y}{\lambda}\right).
\]
It therefore has full rank, and so does the unnormalized physical jet
matrix.  Thus at least one \(8k\times8k\) minor of
\(\mathcal J_{\lambda,\boldsymbol x}\) is not identically zero.

In a Fourier basis of \(E_\lambda^+\), every such minor is a real-analytic
function of \(\boldsymbol x\).  The rank-deficient set is the zero set of the
sum of the squares of all maximal minors.  This sum is nonzero by the
preceding construction, so its zero set is a proper real-analytic subset,
which has measure zero and empty interior.  The standard dimension bound for
the zero set of a nontrivial real-analytic function on the \(3k\)-dimensional
configuration manifold gives
\(\dim_{\mathrm H}\Sigma_{\lambda,k}\le3k-1\).

Surjectivity gives a field with jets \((0,H_\ell)\).  At a zero of a Beltrami
field, the curl equation makes the antisymmetric part of its gradient vanish;
hence the prescribed symmetric gradient is the full gradient.
\end{proof}

\begin{proposition}[Frequency-independent genericity and stable right inverses]
\label{prop:multipoint-stability}
Let \(\lambda_0(k)\) be a threshold in
Theorem~\ref{thm:multipoint-jet}, and set
\[
  \mathcal G_k
  =
  \bigcap_{\substack{\lambda\ge\lambda_0(k)\\ \lambda\ {\rm odd}}}
  \left(
    \operatorname{Conf}_k(\T^3)\setminus\Sigma_{\lambda,k}
  \right).
\]
Then \(\mathcal G_k\) is a frequency-independent full-measure residual
subset of \(\operatorname{Conf}_k(\T^3)\), and
\(\mathcal J_{\lambda,\boldsymbol x}\) is surjective for every
\(\boldsymbol x\in\mathcal G_k\) and every odd
\(\lambda\ge\lambda_0(k)\).

Fix one such \(\lambda\), equip \(E_\lambda^+\) with its \(L^2\) inner
product, and equip the jet target with its Euclidean--Frobenius inner product.
On
\(\mathcal G_{\lambda,k}:=\operatorname{Conf}_k(\T^3)
\setminus\Sigma_{\lambda,k}\), the minimum-norm right inverse
\[
  \mathcal R_{\lambda,\boldsymbol x}
  =
  \mathcal J_{\lambda,\boldsymbol x}^*
  \left(
    \mathcal J_{\lambda,\boldsymbol x}
    \mathcal J_{\lambda,\boldsymbol x}^*
  \right)^{-1}
\]
depends real analytically on \(\boldsymbol x\).  Consequently, for every
compact \(K\subset\mathcal G_{\lambda,k}\), there is a finite constant
\(C_{K,\lambda}\) such that
\[
  \norm{\mathcal R_{\lambda,\boldsymbol x}\boldsymbol q}_{L^2(\T^3)}
  \le
  C_{K,\lambda}\norm{\boldsymbol q}
  \qquad
  (\boldsymbol x\in K).
\]
\end{proposition}

\begin{proof}
For each odd \(\lambda\ge\lambda_0(k)\), the complement of
\(\Sigma_{\lambda,k}\) is open, dense, and of full measure.  The asserted
properties of their countable intersection follow from the Baire theorem and
countable subadditivity of null sets.  In a fixed \(L^2\)-orthonormal Fourier
basis, the jet matrix is real-analytic in \(\boldsymbol x\).  On the full-rank
set its target Gram matrix is positive definite, so the displayed
Moore--Penrose formula is an analytic right inverse.  Its norm is bounded on
compact subsets.
\end{proof}

\begin{proposition}[Shell and time obstructions]
\label{prop:Beltrami-obstructions}
The generic qualifier in Theorem~\ref{thm:multipoint-jet} cannot be removed
for a fixed Fourier shell.  Moreover, within one curl eigenspace, independent
targets at several times cannot replace the one-terminal-time conclusion of
Corollary~\ref{cor:multipoint-realization}.
\begin{enumerate}
  \item The real eigenspace \(E_{\sqrt2}^+\) has dimension \(12\), and cannot
  interpolate arbitrary zero values and trace-free symmetric first jets at
  two points.
  \item If \(\lambda=\sqrt m\) with \(m\in\mathbb N\), every
  \(W\in E_\lambda^+\) satisfies
  \[
    W(x+\pi\boldsymbol 1)=(-1)^mW(x),
    \qquad
    \nabla W(x+\pi\boldsymbol 1)=(-1)^m\nabla W(x),
    \quad
    \boldsymbol 1=(1,1,1).
  \]
  Thus the jets at such a resonant pair cannot be prescribed independently.
  \item If an unforced Euler or Navier--Stokes solution remains in one fixed
  space \(E_\lambda^+\), then
  \[
    u(t)=u(0)\quad(\nu=0),
    \qquad
    u(t)=e^{-\nu\lambda^2t}u(0)\quad(\nu>0).
  \]
  Consequently the deformation gradients at a fixed zero of \(u(0)\) lie in
  a one-parameter matrix subgroup, so arbitrary targets at two distinct
  times are impossible within this ansatz.
\end{enumerate}
\end{proposition}

\begin{proof}
The shell \(\{n\in\mathbb Z^3:\abs{n}^2=2\}\) has \(12\) lattice points, and
the corresponding real positive-helicity eigenspace has dimension \(12\).
By Lemma~\ref{lem:Beltrami} and translation invariance, evaluation at one
point has rank \(3\), so the subspace of fields vanishing there has
dimension \(9\).  It cannot map onto the \(10\)-dimensional space
\(\Symz(3)^2\), even before imposing vanishing at the second point.

For every \(n\in\mathbb Z^3\) with \(\abs{n}^2=m\),
\[
  n_1+n_2+n_3
  \equiv
  n_1^2+n_2^2+n_3^2
  \equiv m\pmod2.
\]
Every Fourier mode in the shell therefore acquires the same factor
\((-1)^m\) under translation by \(\pi\boldsymbol1\), and differentiation
preserves that factor.

Finally, for \(u(t)\in E_\lambda^+\),
\[
  (u\cdot\nabla)u=\nabla\frac{\abs{u}^2}{2},
  \qquad
  \Delta u=-\lambda^2u.
\]
After applying the Leray projector, the unforced equation is
\(\partial_tu=0\) for Euler and
\(\partial_tu=-\nu\lambda^2u\) for Navier--Stokes.  At a fixed zero, the
velocity gradient has the same scalar time factor, and the deformation
gradient is its matrix exponential integrated in time.
\end{proof}

\begin{corollary}[Simultaneous generic multipoint realization]
\label{cor:multipoint-realization}
Fix \(k\ge1\), \(T>0\), and \(\nu\ge0\).  For every sufficiently large odd
integer \(\lambda\), every
\[
  \boldsymbol x=(x_1,\ldots,x_k)
  \in\operatorname{Conf}_k(\T^3)\setminus\Sigma_{\lambda,k},
\]
and every collection
\[
  F_\ell^*\in\Sym^+(3)\cap\SL(3,\R)
  \quad(1\le\ell\le k),
\]
there is one exact real-analytic periodic unforced Euler solution if
\(\nu=0\), and one such Navier--Stokes solution if \(\nu>0\), for which
\[
  X(x_\ell,t)=x_\ell\quad(0\le t\le T),
  \qquad
  \nabla_aX(x_\ell,T)=F_\ell^*
  \quad(1\le\ell\le k).
\]
\end{corollary}

\begin{proof}
Set \(H_\ell=\log F_\ell^*\in\Symz(3)\), and use
Theorem~\ref{thm:multipoint-jet} to choose one
\(W\in E_\lambda^+\) with
\[
  W(x_\ell)=0,
  \qquad
  \nabla W(x_\ell)=H_\ell
  \quad(1\le\ell\le k).
\]
Put
\[
  c=
  \begin{cases}
    T^{-1},&\nu=0,\\[4pt]
    \displaystyle\frac{\nu\lambda^2}
    {1-e^{-\nu\lambda^2T}},&\nu>0,
  \end{cases}
\]
and define
\[
  u(x,t)=ce^{-\nu\lambda^2t}W(x),
  \qquad
  p(x,t)=-\frac12c^2e^{-2\nu\lambda^2t}\abs{W(x)}^2.
\]
This is an exact unforced solution.  Since every \(x_\ell\) is a zero of
\(W\), it is fixed by the flow.  Along this trajectory,
\[
  \nabla u(x_\ell,t)=ce^{-\nu\lambda^2t}H_\ell,
\]
and hence
\[
  \nabla_aX(x_\ell,t)
  =
  \exp\!\left(
    \int_0^tce^{-\nu\lambda^2r}\,dr\,H_\ell
  \right).
\]
The scalar integral equals \(1\) at \(t=T\), so the last display gives
\(\nabla_aX(x_\ell,T)=e^{H_\ell}=F_\ell^*\) simultaneously.
\end{proof}

\section{Quadratic sensing under volume-preserving congruence}

Let \(\Sym(n)\) be the real symmetric \(n\times n\) matrices, equipped with the
Frobenius inner product, and let
\[
  \Symz(n)=\{B\in\Sym(n):\tr B=0\}.
\]
Unless a subscript is displayed, matrix norms are operator norms; the
Frobenius norm is denoted by \(\norm{\cdot}_{\mathrm F}\).  Pointwise matrix
absolute values are Frobenius norms.
For a direction system
\(\mathcal D=\{d_1,\ldots,d_M\}\subset\sphere^{n-1}\), write
\[
  Q_j=d_j\otimes d_j
\]
and define
\[
  \mathcal A_{\mathcal D}B
  =
  \bigl(\ip{B}{Q_j}\bigr)_{j=1}^M
  =
  \bigl(d_j^\top Bd_j\bigr)_{j=1}^M.
\]
For \(F\in\operatorname{GL}(n)\), the normalized transported system is
\[
  F_\#\mathcal D
  =
  \left\{\tau_j(F):=\frac{Fd_j}{\abs{Fd_j}}\right\}_{j=1}^M.
\]

\begin{definition}
The direction system \(\mathcal D\) is \emph{robustly trace-free sensing} if
\[
  B\longmapsto
  \bigl(\tau_j(F)^\top B\tau_j(F)\bigr)_{j=1}^M
\]
is injective on \(\Symz(n)\) for every \(F\in\SL(n,\R)\).
\end{definition}

\begin{theorem}[Sharp channel classification]\label{thm:classification}
Let \(n\ge2\), \(M\ge1\), and
\(\mathcal D=\{d_1,\ldots,d_M\}\subset\sphere^{n-1}\).  The following are
equivalent:
\begin{enumerate}
  \item \(\mathcal D\) is robustly trace-free sensing;
  \item the rank-one projectors span the full symmetric space:
  \[
    \operatorname{span}\{d_j\otimes d_j:1\le j\le M\}=\Sym(n);
  \]
  \item \(\mathcal A_{\mathcal D}\) is injective on \(\Sym(n)\).
\end{enumerate}
Consequently every robust system has
\[
  M\ge N_n=\frac{n(n+1)}2,
\]
and this lower bound is attained.
\end{theorem}

\begin{proof}
The equivalence of the second and third statements is immediate.
Assume that the projectors span \(\Sym(n)\).  Let \(F\in\SL(n,\R)\) and
\(B\in\Symz(n)\) satisfy
\[
  \tau_j(F)^\top B\tau_j(F)=0
  \qquad\hbox{for every }j.
\]
Set \(K=F^\top BF\).  Then
\[
  d_j^\top Kd_j
  =
  \abs{Fd_j}^2\tau_j(F)^\top B\tau_j(F)
  =0.
\]
The spanning hypothesis gives \(K=0\), hence \(B=0\).  Thus the system is
robust.

Conversely, suppose that the projectors do not span \(\Sym(n)\).  If the
directions do not span \(\R^n\), let \(W=\operatorname{span}\mathcal D\).
Choose nonzero \(w\in W\) and \(z\in W^\perp\), and set
\[
  B=w\otimes z+z\otimes w.
\]
Then \(B\ne0\), \(\tr B=0\), and \(d_j^\top Bd_j=0\) for every \(j\).
Robust sensing already fails at \(F=\Id\).

It remains to treat the case in which the directions span \(\R^n\).  Choose
  \[
  0\ne K\in
  \left(
    \operatorname{span}\{d_j\otimes d_j:1\le j\le M\}
  \right)^\perp.
  \]
The matrix \(K\) is indefinite.  Indeed, if \(K\ge0\), then
\(d_j^\top Kd_j=0\) implies \(K^{1/2}d_j=0\) for all \(j\); since the
directions span, this forces \(K=0\).  The same argument excludes \(K\le0\).

Diagonalizing \(K\), choose positive weights on its positive and negative
eigenspaces so that a positive-definite matrix \(A_0\) satisfies
\(\tr(A_0K)=0\).  After multiplying \(A_0\) by a positive scalar, we may assume
\[
  A>0,\qquad \det A=1,\qquad \tr(AK)=0.
\]
Set \(F=A^{-1/2}\in\SL(n,\R)\) and
\[
  B=F^{-\top}KF^{-1}.
\]
Then \(B\ne0\) and
\[
  \tr B
  =
  \tr(F^{-1}F^{-\top}K)
  =
  \tr(AK)=0.
\]
Moreover,
\[
  \tau_j(F)^\top B\tau_j(F)
  =
  \frac{d_j^\top Kd_j}{\abs{Fd_j}^2}
  =0.
\]
Thus robust sensing fails.

The dimension bound follows because \(\dim\Sym(n)=N_n\).  It is attained by
the \(N_n\) directions
\[
  e_i,\qquad
  \frac{e_i+e_j}{\sqrt2}\quad(1\le i<j\le n),
\]
whose projectors recover all diagonal and off-diagonal entries of a symmetric
matrix.
\end{proof}

For quantitative estimates, define
\begin{equation}\label{eq:alphaD}
  \alpha_{\mathcal D}
  =
  \inf_{\substack{C\in\Sym(n)\\ \norm{C}_{\mathrm F}=1}}
  \sum_{j=1}^M(d_j^\top Cd_j)^2.
\end{equation}
By compactness of the unit sphere in \(\Sym(n)\),
\(\alpha_{\mathcal D}>0\) if and only if
\(\ker\mathcal A_{\mathcal D}=\{0\}\).  Thus the classification theorem says
that \(\alpha_{\mathcal D}>0\) precisely for robust systems.

\begin{lemma}[Quantitative congruence bound]\label{lem:congruence}
If \(\alpha_{\mathcal D}>0\), then for every \(F\in\operatorname{GL}(n)\) and
every \(B\in\Sym(n)\),
\begin{equation}\label{eq:congruence}
  \norm{B}_{\mathrm F}
  \le
  \alpha_{\mathcal D}^{-1/2}\kappa(F)^2
  \left(
    \sum_{j=1}^M
    \bigl(\tau_j(F)^\top B\tau_j(F)\bigr)^2
  \right)^{1/2},
\end{equation}
where \(\kappa(F)=\norm{F}\norm{F^{-1}}\) is the Euclidean condition number.
\end{lemma}

\begin{proof}
Set \(C=F^\top BF\).  By the definition of \(\alpha_{\mathcal D}\),
\[
  \norm{C}_{\mathrm F}
  \le
  \alpha_{\mathcal D}^{-1/2}
  \left(\sum_{j=1}^M(d_j^\top Cd_j)^2\right)^{1/2}.
\]
Since
\[
  d_j^\top Cd_j
  =
  \abs{Fd_j}^2\tau_j(F)^\top B\tau_j(F),
\]
the right-hand side is at most
\[
  \alpha_{\mathcal D}^{-1/2}\norm{F}^2
  \left(
    \sum_{j=1}^M
    \bigl(\tau_j(F)^\top B\tau_j(F)\bigr)^2
  \right)^{1/2}.
\]
Finally \(B=F^{-\top}CF^{-1}\), so
\[
  \norm{B}_{\mathrm F}
  \le \norm{F^{-1}}^2\norm{C}_{\mathrm F}.
\]
Combining the last two estimates proves \eqref{eq:congruence}.
\end{proof}

\section{The outer-product-optimal six-line system in dimension three}

Let \(\mathcal D=\{d_1,\ldots,d_6\}\subset\sphere^2\).  We optimize the lower
bound \(\alpha_{\mathcal D}\) from \eqref{eq:alphaD} over all six-direction
systems.  The value and equality geometry below are the real
three-dimensional, six-vector specialization of the optimal outer-product
Riesz bounds in \cite{CasazzaPinkhamTuomanen2016}; we include the short proof
to fix normalization and the quantitative constant used later.

\begin{theorem}[Outer-product frame optimum]\label{thm:optimal-six}
For every six unit directions in \(\R^3\),
\begin{equation}\label{eq:optimal-bound}
  \alpha_{\mathcal D}\le\frac45.
\end{equation}
Equality holds if and only if
\begin{equation}\label{eq:ETF}
  \sum_{j=1}^6d_j\otimes d_j=2\Id,
  \qquad
  \abs{d_i\cdot d_j}^2=\frac15
  \quad(i\ne j).
\end{equation}
In the equality case, for
\[
  B=B_0+\frac{\tr B}{3}\Id,
  \qquad B_0\in\Symz(3),
\]
one has the exact identity
\begin{equation}\label{eq:optimal-identity}
  \sum_{j=1}^6(d_j^\top Bd_j)^2
  =
  \frac45\norm{B_0}_{\mathrm F}^2
  +
  \frac23(\tr B)^2.
\end{equation}
\end{theorem}

\begin{proof}
Let
\[
  Q_j=d_j\otimes d_j,
  \qquad
  P_j=Q_j-\frac13\Id\in\Symz(3).
\]
For every \(j\),
\[
  \norm{P_j}_{\mathrm F}^2=\frac23.
\]
Consider the frame operator on the five-dimensional space \(\Symz(3)\),
\[
  \mathcal S_0B=\sum_{j=1}^6\ip{B}{P_j}P_j.
\]
Its trace is
\[
  \tr_{\Symz(3)}\mathcal S_0
  =
  \sum_{j=1}^6\norm{P_j}_{\mathrm F}^2
  =4.
\]
Hence, testing the Rayleigh quotient of the full sensing frame operator only
on \(\Symz(3)\),
\[
  \alpha_{\mathcal D}
  \le
  \lambda_{\min}(\mathcal S_0)
  \le
  \frac15\tr_{\Symz(3)}\mathcal S_0
  =
  \frac45,
\]
which proves \eqref{eq:optimal-bound}.

Suppose that equality holds.  Then the compression of the full sensing frame
operator to \(\Symz(3)\) has smallest eigenvalue at least \(4/5\).  The trace
identity forces
\[
  \mathcal S_0=\frac45\Id_{\Symz(3)}.
\]
Relative to
\[
  \Sym(3)=\operatorname{span}\{\Id/\sqrt3\}\oplus\Symz(3),
\]
the trace-direction diagonal entry of the full frame operator equals \(2\).
If its off-diagonal block had a nonzero component \(b\), the restriction to
the corresponding two-dimensional subspace would have matrix
\[
  \begin{pmatrix}
    2&\abs b\\
    \abs b&4/5
  \end{pmatrix},
\]
whose smaller eigenvalue is strictly below \(4/5\).
Therefore the off-diagonal block vanishes, which is equivalent to
\[
  \sum_{j=1}^6P_j=0
  \quad\Longleftrightarrow\quad
  \sum_{j=1}^6Q_j=2\Id.
\]

Let \(T:\R^6\to\Symz(3)\) be the synthesis map \(Te_j=P_j\).  We have
\[
  TT^*=\frac45\Id_{\Symz(3)}
\]
and \(\ker T=\operatorname{span}\{(1,\ldots,1)\}\).  Consequently
\[
  T^*T
  =
  \frac45\left(I_6-\frac16\mathbf 1\mathbf 1^\top\right).
\]
For \(i\ne j\),
\[
  \ip{P_i}{P_j}=-\frac{2}{15}.
\]
But
\[
  \ip{P_i}{P_j}
  =
  \abs{d_i\cdot d_j}^2-\frac13,
\]
which gives \(\abs{d_i\cdot d_j}^2=1/5\).

Conversely, \eqref{eq:ETF} gives
\[
  \ip{P_i}{P_j}
  =
  \begin{cases}
    2/3,&i=j,\\
    -2/15,&i\ne j.
  \end{cases}
\]
Thus \(P_1,\ldots,P_6\) form a regular simplex and
\(\mathcal S_0=(4/5)\Id\).  The tightness identity
\(\sum_jP_j=0\) eliminates the mixed trace--trace-free term, so expanding
\(B=B_0+(\tr B)\Id/3\) proves \eqref{eq:optimal-identity}.  Since
\[
  \norm{B}_{\mathrm F}^2
  =
  \norm{B_0}_{\mathrm F}^2+\frac13(\tr B)^2,
\]
we have
\[
  \sum_{j=1}^6(d_j^\top Bd_j)^2
  -
  \frac45\norm{B}_{\mathrm F}^2
  =
  \frac25(\tr B)^2
  \ge0.
\]
Taking a unit trace-free \(B\) gives equality; hence
\(\alpha_{\mathcal D}=4/5\).
\end{proof}

\begin{corollary}[Icosahedral realization]\label{cor:icosahedron}
Let \(\varphi=(1+\sqrt5)/2\) and normalize the six vectors
\[
  (0,1,\varphi),\ (0,1,-\varphi),\
  (1,\varphi,0),\ (1,-\varphi,0),\
  (\varphi,0,1),\ (\varphi,0,-1).
\]
The resulting six unoriented directions satisfy \eqref{eq:ETF} and attain
\(\alpha_{\mathcal D}=4/5\).
\end{corollary}

\begin{proof}
The squared norm of every displayed vector is \(1+\varphi^2\).  Using
\(\varphi^2=\varphi+1\), direct calculation gives
\[
  \sum_{j=1}^6d_j\otimes d_j=2\Id
  \quad\hbox{and}\quad
  \abs{d_i\cdot d_j}^2=\frac15
  \quad(i\ne j).
\]
\end{proof}

\subsection{Dynamical sharpness of the channel count}

\begin{corollary}[Dynamical sharpness of the channel count]
\label{cor:dynamic-sharpness}
Every direction system in \(\R^3\) with at most five channels loses
trace-free sensing at a deformation geometry attained by an exact smooth
periodic Euler flow and, for every \(\nu>0\), by an exact smooth periodic
Navier--Stokes flow at a material point.
\end{corollary}

\begin{proof}
If the system already fails at \(F=\Id\), use the zero solution.  Otherwise,
the proof of Theorem~\ref{thm:classification} supplies
\(F_*=A^{-1/2}\in\Sym^+(3)\cap\SL(3,\R)\) at which trace-free sensing fails.
Apply Theorem~\ref{thm:realization}.
\end{proof}

\begin{corollary}[Simultaneous dynamic sharpness]
Fix \(T>0\) and finitely many direction systems
\(\mathcal D_1,\ldots,\mathcal D_k\subset\sphere^2\), each with at most five
channels.  For every sufficiently large odd \(\lambda\) and every
\[
  \boldsymbol x=(x_1,\ldots,x_k)
  \in\operatorname{Conf}_k(\T^3)\setminus\Sigma_{\lambda,k},
\]
one exact periodic unforced Euler
solution, and for every \(\nu>0\) one exact periodic unforced Navier--Stokes
solution, can make \(\mathcal D_\ell\) lose trace-free sensing at \(x_\ell\)
at one common prescribed terminal time, simultaneously for
\(1\le\ell\le k\).
\end{corollary}

\begin{proof}
For each \(\mathcal D_\ell\), take \(F_\ell^*=\Id\) if sensing already fails
there.  Otherwise, the construction in the proof of
Theorem~\ref{thm:classification} gives
\(F_\ell^*=A_\ell^{-1/2}\in\Sym^+(3)\cap\SL(3,\R)\) at which sensing fails.
Apply
Corollary~\ref{cor:multipoint-realization}.
\end{proof}

\section{Continuation consequences and the viscous frontier}

\subsection{Six-channel Navier--Stokes strain characterization}

Consider the incompressible Navier--Stokes equations on \(\T^3\),
\begin{equation}\label{eq:NS}
  \partial_tu+u\cdot\nabla u+\nabla p=\nu\Delta u,
  \qquad \nabla\cdot u=0,
\end{equation}
with \(\nu>0\).  Let \(0<T<\infty\), \(s>5/2\), and let
\[
  u\in
  C([0,T);H^s_\sigma(\T^3))
  \cap
  L^2_{\mathrm{loc}}([0,T);H^{s+1}_\sigma(\T^3))
\]
be a strong solution.  Set
\[
  S(u)=\frac12(\nabla u+\nabla u^\top),
  \qquad
  \omega=\nabla\times u,
  \qquad
  \omega_0=\nabla\times u(\cdot,0),
\]
and let \(X\) be the Lagrangian flow, with \(F=\nabla_aX\).  For a robust
direction system
\(\mathcal D=\{d_1,\ldots,d_M\}\), define
\begin{align}
  q_j(a,t)
  &=
  \frac{d}{dt}\log\abs{F(a,t)d_j},\\
  g_{\mathcal D}(t)
  &=
  \esssup_{a\in\T^3}
  \left(\sum_{j=1}^Mq_j(a,t)^2\right)^{1/2},\\
  \mathfrak A_{\mathcal D}(t)
  &=
  \int_0^tg_{\mathcal D}(r)\,dr
  \qquad(0\le t<T),
\end{align}
and set
\[
  \mathfrak A_{\mathcal D}(T)
  =
  \lim_{t\uparrow T}\mathfrak A_{\mathcal D}(t)
  \in[0,\infty].
\]

\begin{theorem}[Finite-channel strain characterization]\label{thm:NS-observability}
Let \(\mathcal D\subset\sphere^2\) satisfy the equivalent conditions in
Theorem~\ref{thm:classification}.  Then
\begin{equation}\label{eq:action-equivalence}
  \mathfrak A_{\mathcal D}(T)<\infty
  \quad\Longleftrightarrow\quad
  \int_0^T\norm{S(u)(t)}_{L^\infty(\T^3)}\,dt<\infty.
\end{equation}
Consequently \(u\) extends in \(H^s\) beyond \(T\) if and only if
\(\mathfrak A_{\mathcal D}(T)<\infty\).

For the icosahedral six-direction system,
\begin{equation}\label{eq:ico-quantitative}
  \norm{S(u)(t)}_{L^\infty}
  \le
  27\sqrt{\frac54}\,
  e^{6\mathfrak A_{\mathcal D}(t)}g_{\mathcal D}(t).
\end{equation}
\end{theorem}

\begin{proof}
The stretching identity \eqref{eq:stretch} gives
\[
  g_{\mathcal D}(t)
  \le
  \sqrt M\,\norm{S(u)(t)}_{L^\infty},
\]
which proves one implication.

Conversely, suppose \(\mathfrak A_{\mathcal D}(T)<\infty\).  For almost every
\(a\), every \(0\le t<T\), and every \(j\),
\[
  \abs{\log\abs{F(a,t)d_j}}
  \le \mathfrak A_{\mathcal D}(t).
\]
Since the directions span \(\R^3\), their frame operator has a positive lower
bound.  Hence
\[
  \norm{F(a,t)}
  \le C_{\mathcal D}e^{\mathfrak A_{\mathcal D}(t)}.
\]
Because \(\det F=1\), the singular values give
\[
  \norm{F(a,t)^{-1}}
  \le \norm{F(a,t)}^2,
  \qquad
  \kappa(F(a,t))^2
  \le C_{\mathcal D}e^{6\mathfrak A_{\mathcal D}(t)}.
\]
Apply Lemma~\ref{lem:congruence} to
\[
  B=S(u)(X(a,t),t).
\]
Using \eqref{eq:stretch} and the fact that \(X(\cdot,t)\) preserves volume,
\[
  \esssup_a\abs{S(u)(X(a,t),t)}
  =
  \norm{S(u)(t)}_{L^\infty}.
\]
Taking the essential supremum in \(a\) and then integrating in time yields
\[
  \int_0^T\norm{S(u)(t)}_{L^\infty}\,dt
  \le
  C_{\mathcal D}
  e^{6\mathfrak A_{\mathcal D}(T)}
  \mathfrak A_{\mathcal D}(T)
  <\infty.
\]

For the icosahedral system,
\[
  \sum_{j=1}^6d_j\otimes d_j=2\Id,
\]
and therefore
\[
  2\norm{F}_{\mathrm F}^2
  =
  \sum_{j=1}^6\abs{Fd_j}^2
  \le 6e^{2\mathfrak A_{\mathcal D}(t)}.
\]
Thus \(\norm F\le\sqrt3e^{\mathfrak A_{\mathcal D}(t)}\),
\(\norm{F^{-1}}\le3e^{2\mathfrak A_{\mathcal D}(t)}\), and
\(\kappa(F)^2\le27e^{6\mathfrak A_{\mathcal D}(t)}\).  Combining this with
\(\alpha_{\mathcal D}=4/5\) in Lemma~\ref{lem:congruence} proves
\eqref{eq:ico-quantitative}.

Finally, Kato's inequality applied to the vorticity equation
\[
  \partial_t\omega+u\cdot\nabla\omega-\nu\Delta\omega=S(u)\omega
\]
gives, in the distributional sense,
\[
  (\partial_t+u\cdot\nabla-\nu\Delta)\abs\omega
  \le
  \norm{S(u)(t)}_{L^\infty}\abs\omega.
\]
The maximum principle therefore yields
\[
  \norm{\omega(t)}_{L^\infty}
  \le
  \norm{\omega_0}_{L^\infty}
  \exp\!\left(
    \int_0^t\norm{S(u)(r)}_{L^\infty}\,dr
  \right).
\]
Thus \(\omega\in L^1(0,T;L^\infty)\), and the standard
Beale--Kato--Majda--Kato--Ponce continuation estimate
\cite{BealeKatoMajda1984,KatoPonce1988,BeiraoDaVeiga1995}
gives an \(H^s\) extension beyond \(T\).  Conversely, if the solution extends
in \(H^s\) beyond \(T\), Sobolev embedding makes its strain bounded on a
larger compact time interval, giving the reverse necessity.
\end{proof}

\begin{remark}
Theorem~\ref{thm:NS-observability} is an exact Lagrangian representation of a
classical continuation action.  It is not a proof of global regularity and
does not weaken the Eulerian strain hypothesis.
\end{remark}

\subsection{A one-sided fixed-label criterion for three-dimensional Euler}

Let \(0<T<\infty\), \(s>5/2\), and let
\[
  u\in C([0,T);H^s(\T^3))
  \cap C^1([0,T);H^{s-1}(\T^3))
\]
be a solution of
\begin{equation}\label{eq:Euler}
  \partial_tu+u\cdot\nabla u+\nabla p=0,
  \qquad
  \nabla\cdot u=0.
\end{equation}
Set
\[
  S(u)=\frac12(\nabla u+\nabla u^\top),
  \qquad
  \omega=\nabla\times u,
  \qquad
  \omega_0=\nabla\times u(\cdot,0).
\]
Let \(X\) be its Lagrangian flow, \(F=\nabla_aX\), and let
\(\mathcal E=\{e_1,\ldots,e_M\}\subset\sphere^2\) span \(\R^3\).  Set
\[
  \beta_{\mathcal E}
  =
  \lambda_{\min}
  \left(\sum_{j=1}^Me_j\otimes e_j\right)>0
\]
and
\[
  q_j(a,t)
  =
  \partial_t\log\abs{F(a,t)e_j}
  =
  \tau_j(a,t)^\top S(u)(X(a,t),t)\tau_j(a,t),
  \qquad
  \tau_j=\frac{Fe_j}{\abs{Fe_j}}.
\]
For \(0<t<T\), define the one-sided fixed-trajectory action
\begin{equation}\label{eq:Euler-positive-action}
  \mathfrak P_{\mathcal E}(t)
  =
  \esssup_{a\in\T^3}
  \int_0^t
  \left(
    \sum_{j=1}^M[q_j(a,r)]_+^2
  \right)^{1/2}dr,
\end{equation}
and set
\[
  \mathfrak P_{\mathcal E}(T)
  =
  \lim_{t\uparrow T}\mathfrak P_{\mathcal E}(t)
  \in[0,\infty].
\]

\begin{theorem}[One-sided fixed-trajectory Euler criterion]
\label{thm:Euler-fixed-trajectory}
The solution \(u\) extends in \(H^s\) beyond \(T\) if and only if
\[
  \mathfrak P_{\mathcal E}(T)<\infty.
\]
More precisely, for \(0<t<T\),
\begin{equation}\label{eq:Euler-vorticity-bound}
  \norm{\omega(t)}_{L^\infty}
  \le
  \beta_{\mathcal E}^{-1/2}\sqrt M\,
  e^{\mathfrak P_{\mathcal E}(t)}
  \norm{\omega_0}_{L^\infty}.
\end{equation}
Three directions suffice.  Three is the minimum cardinality for this uniform
frame-based control of arbitrary deformation gradients.  Among three unit
directions, \(\beta_{\mathcal E}\le1\), with equality if and only if the
directions form an orthonormal basis.
\end{theorem}

\begin{proof}
For almost every \(a\), every \(0<t<T\), and each \(j\),
\[
  \log\abs{F(a,t)e_j}
  =
  \int_0^tq_j(a,r)\,dr
  \le
  \int_0^t[q_j(a,r)]_+\,dr
  \le
  \mathfrak P_{\mathcal E}(t).
\]
Writing \(G_{\mathcal E}=\sum_je_j\otimes e_j\), the inequality
\(G_{\mathcal E}\ge\beta_{\mathcal E}\Id\) gives
\[
  \beta_{\mathcal E}\norm{F(a,t)}_{\mathrm F}^2
  \le
  \sum_{j=1}^M\abs{F(a,t)e_j}^2
  \le
  M e^{2\mathfrak P_{\mathcal E}(t)}.
\]
The Euler Cauchy formula
\[
  \omega(X(a,t),t)=F(a,t)\omega_0(a)
\]
therefore yields \eqref{eq:Euler-vorticity-bound}.  If
\(\mathfrak P_{\mathcal E}(T)<\infty\), then
\[
  \int_0^T\norm{\omega(t)}_{L^\infty}\,dt
  \le
  T\beta_{\mathcal E}^{-1/2}\sqrt M\,
  e^{\mathfrak P_{\mathcal E}(T)}
  \norm{\omega_0}_{L^\infty}<\infty,
\]
and the Beale--Kato--Majda criterion
\cite{BealeKatoMajda1984,KatoPonce1988} gives continuation beyond \(T\).

Conversely, if \(u\) extends in \(H^s\) beyond \(T\), then
\[
  \mathfrak P_{\mathcal E}(T)
  \le
  \sqrt M\int_0^T\norm{S(u)(t)}_{L^\infty}\,dt<\infty.
\]

Finally, a spanning family in \(\R^3\) contains at least three vectors.  If
\(M=3\) and the vectors are unit, then
\[
  \beta_{\mathcal E}
  \le
  \frac13\tr
  \left(\sum_{j=1}^3e_j\otimes e_j\right)=1.
\]
Equality holds exactly when the frame operator is the identity, equivalently
when the three directions form an orthonormal basis.
\end{proof}

\begin{corollary}[Deterministic finite-line Cauchy criterion]
\label{cor:Euler-net-lines}
Define
\[
  \mathcal L_{\mathcal E}^{\rm Eul}(t)
  =
  \esssup_{a\in\T^3}
  \left(
    \sum_{j=1}^M\abs{F(a,t)e_j}^2
  \right)^{1/2}.
\]
Then
\[
  \norm{\omega(t)}_{L^\infty}
  \le
  \beta_{\mathcal E}^{-1/2}
  \norm{\omega_0}_{L^\infty}
  \mathcal L_{\mathcal E}^{\rm Eul}(t),
\]
and \(u\) extends in \(H^s\) beyond \(T\) if and only if
\[
  \int_0^T\mathcal L_{\mathcal E}^{\rm Eul}(t)\,dt<\infty.
\]
\end{corollary}

\begin{proof}
The frame bound and the Euler Cauchy formula give the displayed estimate.
Its time integral and the Beale--Kato--Majda criterion prove sufficiency.
If the solution extends, Gronwall's inequality applied to
\(\partial_tF=(\nabla u)(X,t)F\) makes
\(\mathcal L_{\mathcal E}^{\rm Eul}\) bounded on \([0,T]\), proving
necessity.
\end{proof}

\begin{remark}
The criterion uses only positive stretching and places the essential
supremum over material labels outside the time integral; quantitatively,
\[
  \mathcal L_{\mathcal E}^{\rm Eul}(t)
  \le
  \sqrt M\,e^{\mathfrak P_{\mathcal E}(t)}.
\]
Its
three-direction minimality concerns uniform finite-frame control of an
arbitrary deformation gradient, not all conceivable PDE-specific Euler
criteria.  The PDE-specific input is the Euler Cauchy formula; therefore the
result is not contradicted by the purely kinematic moving-pulse construction
in Theorem~\ref{thm:kinematic-firewall}.
\end{remark}

\subsection{A three-line stochastic Lagrangian continuation criterion}

We record a consequence of the stochastic Cauchy formula that is naturally
adapted to finite material-line observations.  Let \(\nu>0\),
\(0<T<\infty\), \(s>5/2\), and let
\[
  u\in
  C([0,T);H^s_\sigma(\T^3))
  \cap
  L^2_{\mathrm{loc}}([0,T);H^{s+1}_\sigma(\T^3))
\]
be a strong solution of \eqref{eq:NS}.  Set
\[
  S(u)=\frac12(\nabla u+\nabla u^\top),
  \qquad
  \omega=\nabla\times u,
  \qquad
  \omega_0=\nabla\times u(\cdot,0).
\]
Let \(W\) be a standard
three-dimensional Brownian motion and let
\[
  dX_t(a)
  =
  u(X_t(a),t)\,dt+\sqrt{2\nu}\,dW_t,
  \qquad
  X_0(a)=a.
\]
We regard \(W_t\) as a spatially uniform Brownian translation of the torus,
and all expectations below are with respect to \(W\).  Since the noise is
independent of the label,
\[
  dJ_t(a)
  =
  (\nabla u)(X_t(a),t)J_t(a)\,dt,
  \qquad
  \det J_t(a)
  =
  \exp\!\left(
    \int_0^t(\nabla\cdot u)(X_r(a),r)\,dr
  \right)
  =
  1.
\]
Moreover,
\[
  Y_t(a)=X_t(a)-\sqrt{2\nu}\,W_t
\]
is the flow of a random \(C^1\) time-dependent vector field.  Thus, pathwise,
\(X_t\) is a volume-preserving \(C^1\) diffeomorphism.  Write
\[
  A_t=X_t^{-1},
  \qquad
  J_t(a)=\nabla_aX_t(a).
\]
For fixed \(d\in\sphere^2\),
\begin{align*}
  r_d(a,t)
  &:=
  \log\abs{J_t(a)d}
  =
  \int_0^t
  \bigl(\tau_d^{\rm st}(a,s)\bigr)^\top
  S(u)(X_s(a),s)
  \tau_d^{\rm st}(a,s)\,ds,\\
  \tau_d^{\rm st}(a,s)
  &:=
  \frac{J_s(a)d}{\abs{J_s(a)d}},
\end{align*}
pathwise in the Brownian realization.

Let \(\mathcal E=\{e_1,\ldots,e_M\}\subset\sphere^2\) span \(\R^3\), and set
\[
  \beta_{\mathcal E}
  =
  \lambda_{\min}
  \left(\sum_{j=1}^M e_j\otimes e_j\right)>0.
\]
Define the stochastic finite-line action
\begin{equation}\label{eq:stochastic-action}
  \mathcal L_{\mathcal E}(t)
  =
  \esssup_{x\in\T^3}
  \mathbb E
  \left[
    \left(
      \sum_{j=1}^M
      e^{2r_{e_j}(A_t(x),t)}
    \right)^{1/2}
  \right].
\end{equation}

\begin{theorem}[Stochastic finite-line criterion]
\label{thm:stochastic-lines}
Under the preceding assumptions, let
\(\mathcal E\subset\sphere^2\) span \(\R^3\).  Then
\begin{equation}\label{eq:stochastic-continuation}
  u\ \hbox{extends in }H^s\hbox{ beyond }T
  \quad\Longleftrightarrow\quad
  \int_0^T\mathcal L_{\mathcal E}(t)\,dt<\infty.
\end{equation}
More precisely,
\begin{equation}\label{eq:stochastic-vorticity-bound}
  \norm{\omega(t)}_{L^\infty}
  \le
  \beta_{\mathcal E}^{-1/2}
  \norm{\omega_0}_{L^\infty}
  \mathcal L_{\mathcal E}(t).
\end{equation}
Three fixed initial directions are sufficient, and three is the smallest
number whose fixed-direction line lengths give uniform pathwise frame control
of every deformation gradient in \(\SL(3,\R)\).
\end{theorem}

\begin{proof}
The stochastic Cauchy formula \cite{ConstantinIyer2008} is
\[
  \omega(x,t)
  =
  \mathbb E
  \left[
    J_t(A_t(x))\,
    \omega_0(A_t(x))
  \right].
\]
For smooth data this is the classical formula; the stated \(H^s\) case
follows by approximation and stability of the \(C^1\) stochastic flows.
Consequently, with
\(
Z=J_t(A_t(x))\omega_0(A_t(x))
\),
\[
  \abs{\omega(x,t)}
  =
  \abs{\mathbb EZ}
  \le
  \mathbb E\abs Z
  \le
  \norm{\omega_0}_{L^\infty}
  \mathbb E\norm{J_t(A_t(x))}_{\mathrm{op}}.
\]
The lower frame bound gives, for every \(J\in\R^{3\times3}\),
\[
  \norm{J}_{\mathrm F}^2
  \le
  \beta_{\mathcal E}^{-1}
  \sum_{j=1}^M\abs{Je_j}^2.
\]
Indeed, if \(G_{\mathcal E}=\sum_je_j\otimes e_j\), then
\[
  \sum_j\abs{Je_j}^2
  =
  \tr(J^\top JG_{\mathcal E})
  \ge
  \beta_{\mathcal E}\norm{J}_{\mathrm F}^2.
\]
Since \(\norm J_{\mathrm{op}}\le\norm J_{\mathrm F}\), applying this estimate
pathwise with \(J=J_t(A_t(x))\), using
\(\abs{J_te_j}=e^{r_{e_j}}\), then taking expectation and only afterward the
essential supremum in \(x\), proves
\eqref{eq:stochastic-vorticity-bound}.  Thus finiteness of the right-hand
side of \eqref{eq:stochastic-continuation} implies
\[
  \int_0^T\norm{\omega(t)}_{L^\infty}\,dt<\infty.
\]
The Beale--Kato--Majda--Kato--Ponce continuation mechanism
\cite{BealeKatoMajda1984,KatoPonce1988} yields extension beyond \(T\).

Conversely, if the solution extends beyond \(T\), then \(\nabla u\) is bounded
on a larger compact time interval.  The equation
\[
  dJ_t=(\nabla u)(X_t,t)J_t\,dt
\]
and Gronwall's inequality give, realization-independently,
\[
  \norm{J_t(a)}_{\mathrm{op}}
  \le
  \exp\left(\int_0^t\norm{\nabla u(r)}_{L^\infty}\,dr\right).
\]
Hence
\[
  \mathcal L_{\mathcal E}(t)
  \le
  \sqrt M
  \exp\left(\int_0^t\norm{\nabla u(r)}_{L^\infty}\,dr\right),
\]
so the action is integrable and the reverse implication follows.

Finally, a finite family has a positive lower frame bound exactly when it
spans \(\R^3\), so at least three directions are necessary and any basis is
sufficient.  If the family does not span, choose a unit vector \(v\) orthogonal
to its span and set
\[
  J_R
  =
  R\,v\otimes v
  +
  R^{-1/2}(\Id-v\otimes v)
  \in\Sym^+(3)\cap\SL(3,\R).
\]
Then all observed lengths tend to zero as \(R\to\infty\), whereas
\(\norm{J_R}=R\).  Thus fewer than three line lengths cannot control arbitrary
volume-preserving deformation.
\end{proof}

\begin{remark}
Theorem~\ref{thm:stochastic-lines} uses net stochastic line elongation rather
than absolute strain accumulation.  It is a finite-sensor consequence of the
stochastic Cauchy formula, in the lineage of the probabilistic vorticity
representations and continuation principles in
\cite{BusnelloFlandoliRomito2005,ConstantinIyer2008}, not a new stochastic
representation or a global regularity theorem.  Whether its action is
strictly weaker than classical criteria on the class of Navier--Stokes
solutions is a separate PDE question.
\end{remark}

\section{Sharp limitations of finite material sensing}

\subsection{One-sided sensing is never uniformly robust}

For a trace-free strain tensor, a natural physical observable is its positive
stretching part.  The next result shows that no finite fixed system controls
this quantity pointwise after arbitrary volume-preserving transport.

\begin{theorem}[Finite one-sided no-go]\label{thm:one-sided}
Let \(n\ge2\) and let
\(\mathcal D=\{d_1,\ldots,d_M\}\subset\sphere^{n-1}\) be finite, with
\(M\ge1\).  There exist
\(F\in\SL(n,\R)\) and \(0\ne B\in\Symz(n)\) such that
\[
  \left(\frac{Fd_j}{\abs{Fd_j}}\right)^\top
  B
  \left(\frac{Fd_j}{\abs{Fd_j}}\right)
  <0
  \qquad(1\le j\le M).
\]
Consequently no estimate of the form
\[
  \norm B
  \le C
  \max_j
  \left[
    \left(\frac{Fd_j}{\abs{Fd_j}}\right)^\top
    B
    \left(\frac{Fd_j}{\abs{Fd_j}}\right)
  \right]_+
\]
can hold uniformly for all \(F\in\SL(n,\R)\) and \(B\in\Symz(n)\).
\end{theorem}

\begin{proof}
Choose \(e\in\sphere^{n-1}\) such that \(e\cdot d_j\ne0\) for all \(j\);
this is possible because a finite union of great spheres does not cover
\(\sphere^{n-1}\).  For \(R>1\), set
\[
  F_R=R^{n-1}e\otimes e+R^{-1}(\Id-e\otimes e).
\]
Then \(F_R\in\Sym^+(n)\cap\SL(n,\R)\), and
\[
  \frac{F_Rd_j}{\abs{F_Rd_j}}
  \longrightarrow
  \operatorname{sgn}(e\cdot d_j)e
  \qquad(R\to\infty).
\]
Take
\[
  B=\Id-ne\otimes e.
\]
This matrix is nonzero and trace-free, with \(e^\top Be=1-n<0\).
For all sufficiently large \(R\), every transported observation is strictly
negative.
\end{proof}

\begin{corollary}[Euler and Navier--Stokes accessible one-sided degeneracy]
For every \(p\in\T^3\) and \(T>0\), the deformation \(F\) in
Theorem~\ref{thm:one-sided} may be chosen to occur as the deformation
gradient of an exact smooth periodic Euler solution and of an exact smooth
periodic Navier--Stokes solution at the material point \(p\) and time \(T\).
\end{corollary}

\begin{proof}
The matrices \(F_R\) used in the proof are positive definite and have
determinant one.  Apply Theorem~\ref{thm:realization} and translate its
realizing field from the origin to \(p\).
\end{proof}

\begin{remark}
The corollary concerns the accessibility of the degenerating sensor geometry.
It does not assert that the algebraic witness \(B\) is simultaneously the
strain of the realizing solution at the terminal time.
\end{remark}

\subsection{A kinematic firewall for fixed-trajectory accumulation}

The action in Theorem~\ref{thm:NS-observability} places the supremum over labels
inside the time integral.  The following construction shows that the order
cannot be exchanged using incompressibility alone.

\begin{theorem}[Moving-pulse firewall]
\label{thm:kinematic-firewall}
Let \(T>0\).  There exists a divergence-free velocity field
\[
  u\in C^\infty([0,T)\times\T^3)
\]
with a volume-preserving Lagrangian flow \(X\).  Write
\[
  F(a,t)=\nabla_aX(a,t).
\]
For every nonempty finite direction system
\(\mathcal D=\{d_1,\ldots,d_M\}\subset\sphere^2\), set
\[
  \tau_j(a,t)=\frac{F(a,t)d_j}{\abs{F(a,t)d_j}},
  \qquad
  q_j(a,t)
  =
  \partial_t\log\abs{F(a,t)d_j}
  =
  \tau_j(a,t)^\top
  S(u)(X(a,t),t)\tau_j(a,t).
\]
Then
\begin{equation}\label{eq:eulerian-diverges}
  \int_0^T\norm{S(u)(t)}_{L^\infty}\,dt=+\infty,
\end{equation}
while
\begin{equation}\label{eq:trajectory-bounded}
  \sup_{a\in\T^3}
  \int_0^T
  \abs{S(u)(X(a,t),t)}\,dt<\infty.
\end{equation}
Consequently, for every nonempty finite direction system,
\[
  \sup_a\int_0^T
  \max_j
  \left|
    \frac{d}{dt}\log\abs{F(a,t)d_j}
  \right|dt<\infty.
\]
More precisely, for every robust direction system \(\mathcal D\), the two
orders of aggregation are separated sharply:
\begin{equation}\label{eq:firewall-channel-separation}
\begin{aligned}
  \int_0^T
  \esssup_{a\in\T^3}
  \left(\sum_jq_j(a,t)^2\right)^{1/2}dt
  &=\infty,\\
  \sup_{a\in\T^3}
  \int_0^T
  \left(\sum_jq_j(a,t)^2\right)^{1/2}dt
  &<\infty.
\end{aligned}
\end{equation}
The constructed velocity field is purely kinematic and is not asserted to
solve Navier--Stokes.
\end{theorem}

\begin{proof}
Choose pairwise disjoint balls \(B_m\Subset\T^3\) with radii tending to zero.
Let \(V\) be a nonzero smooth divergence-free vector field compactly supported
in the unit ball.  By translating and scaling,
\[
  V_m(x)=r_mV\left(\frac{x-x_m}{r_m}\right)
\]
is divergence-free, supported in \(B_m\), and satisfies
\[
  \norm{\nabla V_m}_{L^\infty}
  =
  \norm{\nabla V}_{L^\infty}.
\]
After replacing \(V\) if necessary, assume
\(\norm{S(V)}_{L^\infty}>0\).

Choose disjoint time intervals \(I_m\) accumulating only at \(T\), and choose
nonnegative \(h_m\in C_c^\infty(I_m)\) with
\[
  \int_{I_m}h_m(t)\,dt=1.
\]
Define
\[
  u(x,t)=h_m(t)V_m(x)\quad\hbox{for }t\in I_m,
  \qquad
  u(x,t)=0\quad\hbox{otherwise}.
\]
At every time only one summand is active.  The field is smooth on every compact
subinterval of \([0,T)\), and it is divergence-free.

For each \(m\),
\[
  \int_{I_m}\norm{S(u)(t)}_{L^\infty}\,dt
  =
  \norm{S(V)}_{L^\infty}.
\]
Summing over \(m\) proves \eqref{eq:eulerian-diverges}.

The flow of \(V_m\) is the identity outside \(B_m\) and leaves \(B_m\)
invariant because \(V_m\) is compactly supported there.  Since the balls are
disjoint, every material label is affected by at most one pulse.  Therefore
\[
  \int_0^T
  \abs{S(u)(X(a,t),t)}\,dt
  \le
  \operatorname*{ess\,sup}_x\abs{S(V)(x)}
\]
for every \(a\), proving \eqref{eq:trajectory-bounded}.  The final assertion
follows from the stretching identity \eqref{eq:stretch}.

It remains to verify the sharper channel statement.  During one pulse, the
accumulated phase belongs to \([0,1]\).  The deformation gradients of all
active scaled flows therefore satisfy
\[
  \kappa(F(a,t))\le K_V
\]
with a constant independent of \(m\), \(a\), and \(t\).  If
\(\mathcal D\) is robust, Lemma~\ref{lem:congruence}, applied at
\(x=X(a,t)\), gives
\[
  \left(\sum_jq_j(a,t)^2\right)^{1/2}
  \ge
  \alpha_{\mathcal D}^{1/2}K_V^{-2}
  \abs{S(u)(X(a,t),t)}.
\]
Because \(X(\cdot,t)\) preserves volume, taking the essential supremum in
\(a\), integrating, and summing the pulses proves the first line of
\eqref{eq:firewall-channel-separation}.  Conversely,
\[
  \left(\sum_jq_j(a,t)^2\right)^{1/2}
  \le
  \sqrt M\,\abs{S(u)(X(a,t),t)}.
\]
The one-pulse-per-label property and \eqref{eq:trajectory-bounded} prove the
second line.
\end{proof}

\begin{remark}[The PDE frontier]
Theorem~\ref{thm:kinematic-firewall} isolates the genuinely difficult
viscous strengthening.  The Euler Cauchy invariant supplies exactly the
PDE-specific step used in Theorem~\ref{thm:Euler-fixed-trajectory}; no
deterministic counterpart is presently available for Navier--Stokes:
\[
  \sup_a\int_0^T
  \left(\sum_jq_j(a,t)^2\right)^{1/2}dt<\infty
  \quad\stackrel{?}{\Longrightarrow}\quad
  \hbox{Navier--Stokes continuation}.
\]
Any proof must rule out a Navier--Stokes analogue of the moving-pulse
construction.  Neither finite-dimensional sensing nor volume preservation can
do so.
\end{remark}

\section{Concluding perspective}

The results separate four levels of information:
\begin{enumerate}
  \item static trace-free sensing needs \(N_n-1\) generic scalar observations;
  \item robustness under arbitrary volume-preserving transport needs exactly
  \(N_n\) fixed initial directions;
  \item deterministic Cauchy control for Euler needs three spanning directions
  and admits the one-sided fixed-trajectory action
  \eqref{eq:Euler-positive-action};
  \item the stochastic Navier--Stokes Cauchy formula admits the parallel
  three-direction net-deformation criterion, while its deterministic
  fixed-trajectory analogue remains open.
\end{enumerate}
The congruence-robust classification and its exact unforced deformation
realization form the main new algebraic--dynamical loop.  The global
realization theorem shows that every element of \(\SL(3,\R)\), not only the
positive cone or a neighborhood of the identity, occurs as the one-particle
endpoint derivative of an analytic periodic unforced single-shell Euler or
Navier--Stokes flow.  At low frequency, the explicit eight-dimensional
Beltrami jet gives a local endpoint chart, and the fixed-zero construction
gives every positive-definite target by a closed formula.  The generic
multipoint theorem shows that accessibility is not an isolated one-point
effect, while the low-frequency, resonant, and time-rigidity obstructions
record limits of single-eigenspace interpolation.  The icosahedral constant is
the classical optimal outer-product-frame geometry selected by this loop.
The Euler and stochastic criteria then expose why vector amplification needs
three directions whereas instantaneous symmetric-tensor reconstruction needs
six.  None of these statements resolves global regularity for Euler or
Navier--Stokes; the remaining deterministic viscous fixed-trajectory question
requires a genuinely new PDE mechanism.

\begin{appendices}

\section{The Beltrami jet matrix}\label{app:jet-matrix}

With coefficient order
\[
  (\alpha_1,\beta_1,\alpha_2,\beta_2,
   \alpha_3,\beta_3,\alpha_4,\beta_4)
\]
and target coordinates
\[
  (W_1,W_2,W_3,G_{11},G_{22},G_{12},G_{13},G_{23}),
\]
where \(G=\operatorname{sym}\nabla W(0)\in\Symz(3)\), the jet map in
Lemma~\ref{lem:Beltrami} is represented by
\[
\begin{pmatrix}
0&\frac{\sqrt2}{2}&0&-\frac{\sqrt2}{2}&0&-\frac{\sqrt2}{2}&0&\frac{\sqrt2}{2}\\
0&-\frac{\sqrt2}{2}&0&-\frac{\sqrt2}{2}&1&0&1&0\\
1&0&1&0&0&\frac{\sqrt2}{2}&0&\frac{\sqrt2}{2}\\
-\frac{\sqrt2}{2}&0&\frac{\sqrt2}{2}&0&\frac{\sqrt2}{2}&0&-\frac{\sqrt2}{2}&0\\
\frac{\sqrt2}{2}&0&-\frac{\sqrt2}{2}&0&0&0&0&0\\
0&0&0&0&0&\frac12&0&\frac12\\
0&\frac12&0&\frac12&0&0&0&0\\
0&\frac12&0&-\frac12&0&\frac12&0&-\frac12
\end{pmatrix}.
\]
Its determinant is \(-\sqrt2\).

\end{appendices}

\backmatter

\section*{Statements and Declarations}

\subsection*{Author contributions}
Hao Huang conceived the study, developed the mathematical arguments, carried
out the formal analysis, and wrote and revised the manuscript.

\subsection*{Funding}
No funding was received for conducting this study.

\subsection*{Competing interests}
The author has no relevant financial or non-financial interests to disclose.

\subsection*{Data availability}
No datasets were generated or analysed during the current study.

\bibliography{references}

\end{document}